\documentclass[11pt,a4paper,twoside]{article}

\usepackage[T1]{fontenc}
\usepackage[utf8]{inputenc}
\usepackage{times}

\usepackage[leqno]{amsmath} 
\usepackage{amsthm,amsfonts} 
\usepackage{bm} 

\input xy  \xyoption{all}

\usepackage{geometry}
\usepackage[usenames]{color} 
\usepackage{enumerate}
\usepackage{framed}

\usepackage{hyperref} 
\usepackage{esvect} 

\newcommand{\new}{\color{red}}

\newcounter{zad}

\newcommand{\R}{\mathbb{R}}

\newcommand{\T}{\mathrm{T}}
\newcommand{\lra}{\longrightarrow}
\newcommand{\ra}{\rightarrow}
\newcommand{\eps}{\varepsilon}
\newcommand{\dd}{\operatorname{d}}
\newcommand{\D}{\mathrm{D}}
\newcommand{\pa}{\partial}
\newcommand{\that}{\ |\ }
\def\<#1>{\big\langle #1\big\rangle}
\def\[#1>{\left[ #1\right]}

\newcommand{\du}{\Delta u}
\newcommand{\ddu}[1]{\Delta u^{(#1)}}

\numberwithin{equation}{section} 

\theoremstyle{plain} 
\newtheorem{thm}{Theorem}[section]
\newtheorem{prop}[thm]{Proposition}
\newtheorem{lem}[thm]{Lemma}

\theoremstyle{definition}
\newtheorem{df}[thm]{Definition}
\newtheorem{example}[thm]{Example}

\theoremstyle{remark}
\newtheorem{rem}[thm]{Remark}

\newtheorem{fact}[thm]{Fact}
\newtheorem{hyp}[thm]{Hypothesis}

\newcommand{\End}{\mathrm{End}}
\newcommand{\END}{\bm{\mathrm{End}}}
\newcommand{\image}{\operatorname{Im}}
\newcommand{\g}{\mathfrak{g}}
\newcommand{\spann}{\operatorname{span}_{\R}}

\newcommand{\U}{\Omega} 
	
\newcommand{\q}[1]{q^{(#1)}}

\newcommand{\dotq}[1]{{\dot{q}}^{(#1)}}

\newcommand{\bb}[1]{b^{(#1)}}
\newcommand{\dotbb}[1]{\dot{b}^{(#1)}}
\newcommand{\lin}[1]{\Phi^{(#1)}} 

\newcommand{\dotlin}[1]{{\dot\Phi}^{(#1)}}
\newcommand{\Z}[1]{Z^{(#1)}} 
\newcommand{\Ad}{\operatorname{Ad}}
\newcommand{\AddCoord}[1]{\Psi^{(#1)}}

\title{Higher derivatives of the end-point map of a control-linear system via adapted coordinates}
\author{Michał Jóźwikowski$^{a,b,}$\footnote{Part of this research was conducted during the employment of MJ at the University of Fribourg, financed by the ERC Starting Grant \emph{Geometry of Metric Groups}, grant agreement 713998 GeoMeG. } \footnote{Corresponding author}\\[1ex]
Bartłomiej Sikorski$^{a,}$\footnote{BS was supported by the Ministry of Science and Higher Education of the Republic of Poland project ,,Szkoła Orłów'', project number 500-D110-06-0465160}\\[2ex]
$^a$ Faculty of Mathematics, Informatics and Mechanics\\[1ex]
University of Warsaw\\[2ex]
$^b$ Department of Mathematics\\[1ex]
University of Fribourg}

\begin{document}
\maketitle
\begin{abstract}
We study the end-point map of a  control-linear system in a neighborhood of an arbitrarily chosen trajectory. In particular, we want to calculate the $k$-th order derivative of this map in a given direction. A priori it is a solution of a quite complicated ODE depending on all derivatives of order less or equal $k$. We prove that there exists a special coordinate system adapted to the geometry of the problem, which changes the system of ODEs describing all derivatives  of the end-point map up to order $k$ to equations of a control-affine (non-autonomous control-linear) system, with the direction of derivation playing the role of the  new control. As an application we study controllability criteria for this  system, obtaining first and second-order necessary optimality conditions of sub-Riemannian geodesics.  In particular, for the case of an abnormal minimizer we can interpret \emph{Goh conditions} as non-controllability conditions of this control-affine system for $k=2$. We make a hypothesis that for higher $k$'s its non-controllability corresponds to  recently obtained higher-order analogs of the Goh conditions [Boarotto, Monti,  Palmurella, 2020], [Boarotto, Monti, a Socionovo, 2022].  
\end{abstract}


\section{Introduction}\label{sec:intro}
\paragraph{The problem}

Local properties of the end-point map around a trajectory can be quite significant in some aspects of control theory. They play an important role in phenomena such as local controllability, abnormality, and optimality of trajectories. Let us name a few examples. For optimal control problems openness of the  (extended, i.e. including the costs) end-point map around a trajectory excludes the possibility that this trajectory is optimal. Therefore non-openness of the end-point map gives necessary conditions for optimality. By using the standard version of the Open Mapping Theorem (\emph{OMT}) we get first-order conditions for optimality, i.e equations of the Pontryagin Maximum Principle. These conditions  are also fundamental for classifying extremals of the optimal control problem as \emph{normal} or \emph{abnormal} \cite{Jurdjevic_1997, Agrachev_Sachkov_2004}. In some specific situations, like the study of abnormal sub-Riemannian geodesics, the above conditions are, however, insufficient. Therefore one needs to use more specific versions of the OMT. A degree-two version is a basis of \cite{Agrachev_Sarychev_1996, MJ_SR_deg2}, and recently there has been an attempt of using Sussmann's version of the OMT \cite{Sussmann_OMT} to study sub-Riemannian optimality conditions of any order \cite{Monti_third_order_2020,Boarotto_Monti_Socionovo_2022}.

For a control system on a manifold $M$, whose set of admissible controls $\U$ forms a vector space, perhaps the most natural idea to study the time $t$-end-point map $\End^t:\U\ra M$ is by calculating its Taylor expansion. In this paper  we consider such an expansion for a  control-linear system
\begin{equation}
\tag{$\Lambda$}
    \dot{q}(t)=\sum_i^l u_i(t) X_i(q(t))\ ,
\end{equation}
where $u=(u_1,\hdots,u_l)\in \U \subset \operatorname{Meas}([0,T],\R^l)$ is the control, and  $X_i$, for $i=1,2,\hdots,l$ are linearly independent vector fields on $M$. 
Our goal is to calculate the $k$-th order derivative  of the end-point map at a given control $u\in \U$ in the direction of a control $\du\in \T_u\U\simeq \U$, i.e $\D^k_u\End^t[\du]=\frac{\dd^k}{\dd s^k}\big|_{s=0}\End^t[u+s\cdot\du]$.  At this point, it is worth to remark, that as $\End^{t}$ is a manifold-valued map, such a derivative usually has no geometric sense, although it can be calculated in any chosen local coordinate system. This issue can be easily resolved by using the geometric notion of a $k$-jet, rather than a coordinate-dependent notion of a $k$-th derivative. We comment on this matter  in Remark~\ref{rem:jets}. 

\paragraph{Our results} 

First of all, in Lemma~\ref{lem:b_k}, we were able to derive a hierarchy of ODEs, linear in both $u$ and $\du$, describing the desired derivatives  in a local coordinate system. Our main result is, however, an observation that there exists a time-dependent family of diffeomorphisms (depending on the given control $u$ and the vector fields $X_i$) which transforms the differential equations describing the set of derivatives  $$ (\D_u\End^t[\du],\D^2_u\End^t[\du],\hdots,\D^k_u\End^t[\du])\ ,$$ 
 i.e. the $k$-jet of the curve $s\mapsto \End^t(u+s\cdot\du)$, into a system of ODEs for new variables $(\q{1},\q{2},\hdots,\q{k})$, which is linear in $\du$, and has no explicit dependence of the initial control $u$. 
These transformations are described by a family of multi-linear maps -- their construction and properties are stated as Theorem~\ref{thm:adapted}. For reasons that will be clarified in the next paragraph, we call these newly constructed variables -- \emph{adapted coordinates}. As a corollary, we formulate Theorem~\ref{thm:adapted_q} which states that  $(\q{1}(t,\du),\q{2}(t,\du),\hdots,\q{k}(t,\du))$ -- the solutions of the considered ODE system in adapted coordinates -- are trajectories of a non-autonomous control-linear system (or, as we prefer to see it, a control-affine system in variables $(t,\q{1},\hdots,\q{k})$ with the drift being simply $\pa_t$) with $\du$ playing the role of a linear control. Moreover, this system of ODEs is naturally graded. The above results, i.e. Lemma~\ref{lem:b_k}, Theorem~\ref{thm:adapted}, and Theorem~\ref{thm:adapted_q}, are further generalized as, respectively, Lemma~\ref{lem:b_k_polynomial}, Theorem~\ref{thm:adapted_polynomial} and Theorem~\ref{thm:adapted_q_polynomial} to describe $k$-jets of curves $s\mapsto\End^t(u+s\cdot\ddu{1}+s^2\cdot\ddu{2}+\hdots+s^k\cdot\ddu{k})$, which seem to be more promising from the point of view of applications.  \medskip

The idea behind adapted coordinates is actually quite easy to explain. Note that after fixing a control $u$, system \eqref{eqn:lin_ctrl_syst} evolves according to a time-dependent vector field $X_u(t,q)=\sum_{i=1}^l u_i(t)X_i(q)$. Under mild assumption this evolution produces a time-dependent flow $X^{tt'}_u$ acting on $M$  (being a groupoid rather than a group). Now the tangent map of this flow acts on $\T M$ and, in particular, it gives a distinguished evolution in the tangent spaces $\T_{q(t)} M$ along a given trajectory of the system $q(t)$. Thus after fixing a basis $\{e_1(0),\hdots,e_n(0)\}$ of $\T_{q(0)}M$, the flow spreads it into a family of bases $\{e_1(t),\hdots,e_n(t)\}$, each for every particular tangent space $\T_{q(t)}M$. Adapted coordinates $\q{1}(t,\du)$ are precisely the coordinates of the tangent vector $\D_u\End^t[\du]$ in this new basis. The reason why the evolution equation for $t\mapsto \D_u\End^t[\du]$ looks simpler in these new coordinates is because the natural evolution by the flow of the control vector field $X_u(t,q)$ was included in their construction. In other words, the new coordinates are \emph{adapted} to the chosen evolution of the system. 
\smallskip

Higher-order adapted coordinates $(\q{1},\hdots,\q{k})$ are constructed analogously, by considering the action of the flow of the control vector field on $k$-jets. Thus one should think that $(\q{1},\hdots,\q{k})$ are coordinates in  $\T^k_{q(t)}M$ -- the space of $k$-jets of curves passing through the base point $q(t)$ -- naturally ,,preferred'' by the evolution on $M$ given by the field $X_u(t,q)$. We clarify this intuition in Lemma \ref{lem:adapted_geometric}.  The presence of a natural graded structure on $\T^kM$ \cite{Saunders_1989,Grabowski_Rotkiewicz_2012} is the reason why we observe the presence of the natural grading in the ODE systems in Theorems~\ref{thm:adapted} and~\ref{thm:adapted_q}. The theory is illustrated with a study of a particular system of a generalized Martinet system, running throughout the text (Examples~\ref{ex:1}, \ref{ex:2} and~\ref{ex:3}).
\medskip

We propose two particular applications of the presented theory. First, in Subsection~\ref{ssec:groups}, we calculate the adapted coordinates for an invariant control system on a matrix group. In this case the construction turns out to be quite simple. Moreover, it allows for an easy derivation of some formulas of Le Donne \cite{ELD_new}. We also comment on the case of a general Lie group. 

Later in Subsection~\ref{ssec:sr}, we study second-order optimality conditions for a sub-Riemannian geodesic problem. Our idea is to interpret the assertion of Agrachev-Sarychev Index Lemma \cite{Agrachev_Sarychev_1996} as  non-controllability conditions for a certain control system involving first and second derivatives of the end-point map. However, by Theorem~\ref{thm:adapted_q}, in adapted coordinates $(t,\q{1},\q{2})$ this system is a control-affine system, and we may address the question of its controllability using the results of Sussmann and Jurdjevic \cite{Sussmann_Jurdjevic_1972}. It turns out that the criteria for non-controllability give precisely the \emph{Goh conditions} \cite{Goh_1966} (see Lemmas~\ref{lem:results_deg_2} and \ref{lem:parts}). It seems to us that such an interpretation of Goh conditions was not present in the literature so far, not to mention that (once the Agrachev-Sarychev Index Lemma is known) the proof does not require making any estimates. This observation is a basis of a Hypothesis~\ref{hyp} (backup-ed by some calculations) that higher-order Goh conditions introduced recently in \cite{Monti_third_order_2020,Boarotto_Monti_Socionovo_2022} are a consequence of non-controllability of a control-affine system described by Theorem~\ref{thm:adapted_q_polynomial}.
We also refer to a publication of one of us \cite{MJ_SR_deg2} for further applications of adapted coordinates in sub-Riemannian geometry. Here these were used to study the geometry of second-order approximation of the end-point map around a minimizing trajectory.

\paragraph{A remark about jets}
\begin{rem}\label{rem:jets}
In general, there is a sense to speak about higher derivatives only for maps valued in a vector space. Indeed, consider the end-point map $\End^t:\U\lra M$ at $u\in \Omega$. Let $\phi: M\supset U\lra \R^n$ be local coordinates on $M$  around $p=\End^t(u)$. We may then calculate the Taylor expansion of $\phi\circ \End^t:\U\lra \R^n$, but then, say, the $k$-th order term will not transform well when passing to a new coordinate system. In other words, $\D^k_u\End^t[\du]$ is not a well-defined geometric object for $k\geq 2$. 

There are essentially two ways to deal with this problem. The first one is to restrict our attention to a situation in which a higher derivative makes sense. This strategy has been used  in some Agrachev's works -- see for example \cite{Agrachev_Sachkov_2004}, where the second derivative $\D^2_u\End^t$ is defined on $\ker \D_u\End^t$ and takes values in $\operatorname{coker} \D_u\End^t=\T_{\End^t(u)}M/ \image \D_u\End^t$. In a moment we shall explain why it is so. Similar constructions are present, for instance, in \cite{Monti_third_order_2020,Boarotto_Monti_Socionovo_2022}.

In this paper we prefer a different approach. Instead of being interested in each particular term of the Taylor expansion, we want to consider the whole series up to a term of a given order $k$. This can be formalized in the language of jets as follows (cf. \cite{Saunders_1989}). We say that two curves $\gamma,\gamma':(-\eps,\eps)\ra M$ passing through $p=\gamma(0)=\gamma'(0)\in M$ are \emph{tangent up to order $k$ at $p$}, if in some (and thus any) coordinate system $\phi:U\ra \R^n$ around $p$ we have for all $m=1,2\hdots,k$:
$$\frac{\dd^m}{\dd s^m}\Big|_{s=0}\phi (\gamma(s))=\frac{\dd^m}{\dd s^m}\Big|_{s=0}\phi (\gamma'(s))\ . $$
The relation of the $k$-th order tangency is an equivalence relation and its equivalence classes are called \emph{$k$-jets}. We will denote the space of all $k$-jets on $M$ by $\T^kM$. The assignment of a $k$-jet to its base point $[\gamma]_{\sim_k}\mapsto p=\gamma(0)$ makes $\T^kM\ra M$ a locally trivial bundle. In fact, this is an example of a \emph{graded bundle} in the sense of Grabowski and Rotkiewicz \cite{Grabowski_Rotkiewicz_2012}, meaning that the fibers $\T^k_pM$ are endowed with a canonical action of the multiplicative reals. Actually, the structure of $\T^kM$ is much more specific, namely natural projections $[\gamma]_{\sim i}\mapsto [\gamma]_{\sim (i-1)}$ give rise to the tower of fibrations $\T^kM\ra \T^{k-1}M\ra \hdots\T^2 M\ra \T M\ra M$, where the first level is a vector bundle and each higher level is an affine bundle modeled on $\T M$ (see \cite{Saunders_1989}).  

More specifically, in this paper we will consider the Taylor expansion of the end-point map, which after identifying a manifold $M$ with $\R^n$ by a choice of local coordinates may be written as follows:
\begin{align*}\End^t[u+s\du]\overset{loc}=\End^t[u]+s\cdot \D_{u}\End^t[\du]+s^2\cdot\frac 1{2!} \D^2_{u}\End^t[\du]+\hdots+s^k \cdot\frac 1{k!} \D^k\End^t_{u}[\du]+o(s^k)\ .
\end{align*}
As we mentioned earlier, individual terms $\D^i_u\End^t[\du]$ have no geometric meaning for $i\geq 2$, yet their whole collection $(\D_u\End^t[\du],\D^2_u\End^t[\du],\hdots,\D^k_u\End^t[\du])$, encoding the $k$-jet of a curve $s\mapsto \End^t[u+s
\du]$, does. This can be to some extent seen at the level of ODEs \eqref{eqn:bb_k} which describe the time evolution of this $k$-jet: equation for $\D^k_u\End^t[\du]$ contains terms depending on lower-order derivatives, hence there is no way to separate $\D^k_u\End[\du]$ from the rest of the collection $(\D_u\End^t[\du],\D^2_u\End^t[\du],\hdots,\D^k_u\End^t[\du])$. 
\end{rem}

Finally, note that the concept of a jet allows to explain the understanding of second (and higher -- see \cite{Monti_third_order_2020}) derivatives in the spirit of Argachev. Namely, take a curve $u_s=u+s\cdot u^{(1)}+s^2\cdot u^{(2)}+o(s^2)$ in $\Omega$, and let us calculate the second Taylor expansion of $\End(u_s)$ in some local coordinate system:
\begin{equation}
    \label{eqn:expansion_End_2_degree}
\End(u+s\cdot u^{(1)}+s^2\cdot u^{(2)})\overset{loc}=\End(u)+s\cdot\D_u\End\left[u^{(1)}\right]+s^2\cdot\left(\D_u\End\left[u^{(2)}\right]+\frac 12\D^2_u\End\left[u^{(1)},u^{(1)}\right]\right)+o(s^2)\ .
\end{equation}
Thus to extract the $\D^2_u\End\left[u^{(1)},u^{(1)}\right]$-term from the whole 2-jet $(\D_u\End\left[u^{(1)}\right],\D_u\End\left[u^{(2)}\right]+\D^2_u\End\left[u^{(1)},u^{(1)}\right])$ one needs first to assume that $\D_u\End\left[u^{(1)}\right]=0$ and then quotient out the $\D_u\End\left[u^{(2)}\right]$-term, precisely as in the Agrachev's approach  (where the second derivative is defined for $u^{(1)}\in\ker \D_u\End$, and takes values in $\operatorname{coker} \D_u\End$).

\paragraph{Notation: graded multi-indexes and the polynomial expansion of a composition of maps}

\begin{rem}[Notation convention]\label{multi_ind_not}
Consider a multi-index $\alpha=(a_1,a_2,\hdots,a_k)$ and let us introduce the following notation:
\begin{align*}
    &\text{$|\alpha|:=\sum_{i=1}^k a_i$\ is the \emph{absolute value} of $\alpha$,}\\
    &\text{$w(\alpha):=\sum_{i=1}^k i\cdot a_k$\ is the \emph{weight} of $\alpha$}\\
    &\alpha!:=a_1!\cdot a_2!\cdot\hdots\cdot a_k!\ .
\end{align*}
It is convenient to think that $\alpha$ is graded, with $a_i$ being of weight $i$.  Moreover, for an $|\alpha|$-linear map $\Phi^{|\alpha|}$ and $\vec{b}=(\bb{1},\bb{2},\hdots,\bb{k})$ we define 
$$\Phi^{|\alpha|}\left[\vec{b}^{ \ \alpha}\right]:=\Phi^{|\alpha|}\Big[\underbrace{\bb{1},\hdots,\bb{1}}_{a_1},\underbrace{\bb{2},\hdots,\bb{2}}_{a_2},\hdots,\underbrace{\bb{k},\hdots,\bb{k}}_{a_k}\Big]\ .$$
Throughout this paper we understand gradings and weights according to \cite{Grabowski_Rotkiewicz_2012}.
\end{rem}
The above convention allows for an elegant description of the polynomial expansion of a composition of a function and a curve.
\begin{lem}[Faa di Bruno]\label{lem:composition}
Consider a smooth curve $\gamma:(-\eps,\eps)\lra \R^n$, whose polynomial expansion at $s=0$ reads as
$$\gamma(s)=\gamma(0)+s\cdot \gamma^{(1)}+s^2\cdot\gamma^{(2)}+\hdots+s^k\cdot \gamma^{(k)}+o(s^k)\ ,$$
and let $f:\R^n\lra\R$ be a smooth function. Then their composition $F:=f(\gamma):(-\eps,\eps)\lra \R$ expands as
$$F(s)=F(0)+s\cdot F^{(1)}+s^2\cdot F^{(2)}+\hdots s^k\cdot F^{(k)}+o(s^k)\ ,$$
where 
\begin{equation}
\label{eqn:composition}
    F^{(m)}=\sum_{\alpha,\ w(\alpha)=m}\frac 1{\alpha!}\cdot \D^{|\alpha|}_{\gamma(0)}f\left[\vec{\gamma}^{\ \alpha}\right]\ 
\end{equation}
with $\vec{\gamma}=(\gamma^{(1)},\gamma^{(2)},\hdots,\gamma^{(k)})$.
\end{lem}
We leave the proof as an exercise.

\section{Higher derivatives of the end-point map}\label{sec:derivatives}

\paragraph{A control-linear system} Consider a control-linear system \eqref{eqn:lin_ctrl_syst} on an $n$-dimensional manifold $M$ given by a rank-$l$ distribution  $\mathcal{D}_p=\spann\{X_1(p),\hdots, X_l(p)\}\subset \T_p M$, for all $p\in M$:
\begin{equation}\label{eqn:lin_ctrl_syst} \tag{$\Lambda$}
\dot q(t)=\sum_{i=1}^l u_i(t)X_i(q(t));\quad\text{with the initial condition $q(0)=q_0$,}
\end{equation}
where $u(t)=(u_1(t),\hdots,u_l(t))\in \R^l$ is the control. We assume that $u$ belongs to some space $\U$. Unless specified differently,  we will assume only that $\U$ is a vector subspace of the space of measurable maps $\operatorname{Meas}([0,T],\R^l)$.

\paragraph{The end-point map}
For each $t\in[0,T]$ denote by
$$\End^t:\U\lra M;\qquad \End^t:u\longmapsto q(t)\text{ satisfying \eqref{eqn:lin_ctrl_syst}}\ $$
 the family of the \emph{end-point maps} related with \eqref{eqn:lin_ctrl_syst}. Our goal is to calculate $\D^k_u \End^t[\du]$ -- the $k$-th derivative of $\End^t$ at a given control $u$ in the direction of  $\du\in\T_u\U\simeq \U$, for each $k=1,2,3,\hdots$.
\medskip

To do this let us consider a family of solutions $q_s(t)$ of \eqref{eqn:lin_ctrl_syst} corresponding to the controls $u+s\cdot\du$, i.e.
\begin{equation}
    \label{eqn:q_s_t}
\dot q_s(t)=\sum_i(u_i(t)+s\cdot\du_i(t))X_i(q_s(t));\quad\text{with $q_s(0)= q_0$.}
\end{equation}
The above is an ODE in the sense of Caratheodory  smoothly depending on the parameter $s$.  As such it has a solution smoothly depending on $s$ -- see \cite{Bressan_Piccoli_2004}.  The Taylor expansion of $s\mapsto q_s(t)$ at $s=0$ (is some local coordinate system) gives us
\begin{align*}q_s(t)=&\End^t[u+s\du]\overset{loc}=\\
&\End^t[u]+s\cdot \D_{u}\End^t[\du]+s^2\cdot\frac 1{2!} \D^2_{u}\End^t[\du]+\hdots+s^k \cdot\frac 1{k!} \D^k_{u}\End^t[\du]+o(s^k)=\\
&q(t)+s\cdot \bb{1}(t,\du)+s^2\cdot\bb{2}(t,\du)+\hdots s^k\cdot\bb{k}(t,\du)+o(s^k)\ ,\\
&\text{where}\qquad \bb{m}(t,\du):=\frac 1{m!}\frac{\pa^m}{\pa s^m}\Big|_{s=0} q_s(t)=\frac 1{m!}\D^m_u \End^t[\du].\end{align*}  Actually, in light of Remark~\ref{rem:jets}, there is a sense to speak about higher derivatives $\D^k_u\End^{t}[\du]$ only if we locally identify $M$ with $\R^n$ by means of a particular choice of local coordinates.  Therefore, while reading the remaining part of this article the Reader should remember that the ODEs characterizing curves $\bb{m}(t,\du)$ are derived in a particular coordinate frame.  However, in spite of that, the whole $k$-jet of the curve $s\mapsto \End^t[u+s\cdot\du]$ at $s=0$, i.e.
$$q(t)+s\cdot \bb{1}(t,\du)+s^2\cdot\bb{2}(t,\du)+\hdots s^k\cdot\bb{k}(t,\du)
+o(s^k)$$  is a well-defined geometric notion. We shall refer to it as a \emph{$k$-variation} of $\End^t$ at $u$ in the direction of $\du$.
\smallskip

As $\dotbb{m}(t,\du)=\frac 1{m!}\pa^m_s\big|_{s=0}\dot{q}_s(t)$, for each $m=1,2,\hdots,k$ curves $\bb{m}(t,\du)$ are solutions of a system of ODEs obtained by an $m$-fold differentiation of  \eqref{eqn:q_s_t} at $s=0$. The first few of these equations are easy to calculate: 
\begin{align*}
\dotbb{1}(t,\du)=&\sum_i u_i\D_{q(t)} X_i [\bb{1}]+\sum_i \du_i X_i\Big|_{q(t)}\\
\dotbb{2}(t,\du)=&\sum_i u_i\left(\D_{q(t)} X_i[\bb{2}]+\frac 1{2!} \D^2_{q(t)} X_i[\bb{1},\bb{1}]\right)+\sum_i \du_i \D_{q(t)}X_i[ \bb{1}]\\
\dotbb{3}(t,\du)=&\sum_i u_i\left(\D_{q(t)}X_i[\bb{3}]+ \D^2_{q(t)}X_i[ \bb{1},\bb{2}]+\frac{1}{3!}\D^3_{q(t)}X_i [\bb{1},\bb{1},\bb{1}]\right)+\\
&\sum_i \du_i \left(\D_{q(t)}X_i[\bb{2}]+\frac 1{2!}\D^2_{q(t)}X_i[ \bb{1},\bb{1}]\right)\ .
\end{align*}
Here $q(t)=q_0(t)$ is a trajectory of \eqref{eqn:lin_ctrl_syst} corresponding to the control $u$. A clear inductive pattern begins to be seen. Moreover, the above equations are naturally graded after putting $\deg (\bb{m})=m$ and $\deg(\du_i)=1$. In general, we have the following result

\begin{lem}[the general form of $\bb{m}(t,\du)$s]
\label{lem:b_k}
 Consider the $k$-jet of the curve $s\mapsto \End^t[u+s\cdot\du]$ at $s=0$  given in local coordinates by 
$$\End^t(u+s\cdot\du)\overset{loc}=q(t)+s\cdot \bb{1}(t,\du)+s^2\cdot\bb{2}(t,\du)+\hdots s^k\cdot\bb{k}(t,\du)
+o(s^k)\ .$$
Then curves $\bb{m}(t,\du)$
 are subject to the following  system of ODEs:
\begin{equation}\label{eqn:bb_k} 
\begin{split}\dotbb{m}(t,\du) =& \sum_{i} u_i\left(\sum_{\alpha,\  w(\alpha)=m} \frac 1{\alpha!} \cdot\D^{|\alpha|}_{q(t)}X_i[\vec{b}(t,\du)^\alpha]  \right)
+\\
&\sum_{i} \du_i\left(\sum_{\beta,\  w(\beta)=m-1} \frac 1{\beta !} \cdot\D^{|\beta|}_{q(t)}X_i[\vec{b}(t,\du)^\beta]  \right)
\end{split}
\end{equation}
where we use the notation introduced on page \pageref{multi_ind_not}, i.e. the summation is taken over all multi-indices $\alpha=(a_1,a_2,\hdots)$ of a given weight $w(\alpha):=\sum_i i\cdot a_i$ equal $m$ and $m-1$, respectively, $\D^{|\alpha|}_{q(t)}X_i$ denotes the $|\alpha|$-th derivative of $X_i$ at $q(t)$ understood as a $|\alpha|$-linear map, and $\vec{b}(t,\du)=(\bb{1}(t,\du),\bb{2}(t,\du),\hdots,\bb{k}(t,\du))$.

\end{lem}
\begin{proof}
The $m$-fold differentiation of \eqref{eqn:q_s_t} gives us
\begin{align*}
    \frac{\dd}{\dd t}\left(\pa_s^mq_s(t)\right)=&\pa_s^m\left(\dot{q_s}(t)\right)=\pa_s^m\left[\sum_i \left(u_i(t)+s\cdot \du_i(t)\right)X_i(q_s(t))\right]=\\
    &\sum_i u_i(t)\pa_s^m\left[X_i(q_s(t))\right]+\sum_i \du_i(t)\pa_s^m\left[s\cdot X_i(q_s(t))\right]=\\
    &\sum_i u_i(t)\pa_s^m\left[X_i(q_s(t))\right]+\sum_i \du_i(t)\left\{m\cdot\pa_s^{m-1}\left[X_i(q_s(t))\right]+s\cdot\pa_s^{m}\left[X_i(q_s(t))\right]\right\}
\end{align*}
Now considering the above equality at $s=0$, since $\pa_s^m\big|_0q_s(t)=m!\cdot \bb{m}(t)$ we get
\begin{align*}
    m!\cdot \dotbb{m}(t)=\sum_i u_i(t)\pa_s^m\Big|_0\left[X_i(q_s(t))\right]+m\sum_i \du_i(t)\pa_s^{m-1}\Big|_0\left[X_i(q_s(t))\right]
\end{align*}
Finally $q_s(t)$ expands as
$q_s(t)=q(t)+s\cdot \bb{1}(t)+s^2\cdot\bb{2}(t)+\hdots s^k\cdot\bb{k}(t)+o(s^k)$
so, by the results of Lemma~\ref{lem:composition},
\begin{align*}
    &\pa_s^m\Big|_0\left[X_i(q_s(t))\right]=m! \sum_{\alpha,\  w(\alpha)=m} \frac{1}{\alpha!} \cdot\D^{|\alpha|}_{q(t)}X_i[\vec{b}(t)^\alpha]\qquad\text{and}\\ &\pa_s^{m-1}\Big|_0\left[X_i(q_s(t))\right]=(m-1)! \sum_{\beta,\  w(\beta)=m-1} \frac{1}{\beta!} \cdot\D^{|\beta|}_{q(t)}X_i[\vec{b}(t)^\beta]
\end{align*}
Thus \eqref{eqn:bb_k} holds.
\end{proof}
 Let us comment briefly on the question of the existence of solutions of the system \eqref{eqn:bb_k}. On the one hand, we may argue that by a general result of existence, uniqueness, and parameter-regularity of ODEs in the sense of Caratheodory -- see for example \cite{Bressan_Piccoli_2004} -- curves $\bb{m}(t,\du)$ should be well defined for all $t$ for which $q(t)$ satisfying \eqref{eqn:lin_ctrl_syst} makes sense. We may, however, see this fact from a more technical perspective. Note namely, that if the $(m-1)$-jet $(\bb{1}(t,\du),\hdots,\bb{m-1}(t,\du))$ is known, then by \eqref{eqn:bb_k}, the equation for $\bb{m}(t,\du)$ takes the form of an affine ODE
$$\dotbb{m}=\sum_{i}u_i(t)\D_{q(t)}X_i[\bb{m}]+f(t)\ ,$$
where $f(t)$ is measurable. Thus, whenever one finds a fundamental solution of the homogeneous part of this equation,  
$$\dot b=\sum_{i}u_i(t)\D_{q(t)}X_i[b]\ ,$$
$\bb{m}(t,\du)$ can be derived by the standard Cauchy formula. Yet the fundamental solution of the homogeneous part is simply the tangent map of the time-dependent flow of the control vector field $X_u(t,q)=\sum_{i}u_i(t) X_i(q)$. Thus it is indeed well-defined whenever the flow of $X_u(t,q)$ is well-defined. For a more detailed discussion of time-dependent flows and their tangent maps see \cite[Sec. 2]{MJ_WR_contact_pmp}.

\begin{example}[Generalized Martinet system, part 1.]\label{ex:1}
After \cite{Monti_third_order_2020}, consider a control system in $\R^3\ni(x_1,x_2,x_3)$ given by a pair of vector fields
$$X_1=\frac{\pa}{\pa_{x_1}}\quad\text{and}\quad X_2=(1-x_1)\cdot\frac{\pa}{\pa x_2}+x_1^p\cdot\frac{\pa}{\pa x_3}\ ,$$
where $p>1$ is an integer.
Let $u(t)\equiv(0,1)$ be the control corresponding to a trajectory $q(t)=(0,t,0)$, and take another control $\du(t)=(\du_1(t),\du_2(t))$. Let  us describe curves $\bb{m}(t,\du)$ for $m=1,2,\hdots, p,p+1$ in this setting.

We have $\D^m_{q(t)} X_1\equiv0$ for every $m=1,2,\hdots$, while only non-trivial derivatives of $X_2$ along the considered trajectory are 
$$\D_{q(t)} X_2[\cdot]=-\<\dd x_1,\cdot >\cdot\pa_{x_2}\quad\text{and}\quad \D^p_{q(t)}X_2[\underbrace{\cdot,\cdot,\hdots,\cdot}_p]=p! \<\dd x_1,\cdot>^p\cdot\pa_{x_3}\ .$$
It follows that  equations \eqref{eqn:bb_k} look as follows
\begin{align*}
    \dotbb{1}(t)&=-\bb{1}_1(t)\cdot\pa_{x_2}+\du_1(t) \cdot\pa_{x_1}+\du_2(t)\cdot\pa_{x_2}\\
    \dotbb{2}(t)&=-\bb{2}_1(t)\cdot\pa_{x_2}-\du_2(t)\cdot \bb{1}_1(t)\cdot\pa_{x_2}\\
    \hdots\\
    \dotbb{m}(t)&=-\bb{m}_1(t)\cdot\pa_{x_2}-\du_2(t)
    \cdot \bb{m-1}_1(t)\cdot \pa_{x_2}\\
    \hdots\\
    \dotbb{p}(t)&=-\bb{p}_1(t)\cdot\pa_{x_2}-\du_2(t)\cdot \bb{p-1}_1(t)\cdot\pa_{x_2}+{\left(\bb{1}_1(t)\right)^p\cdot\pa_{x_3}}\\
     \dotbb{p+1}(t)&= -\bb{p+1}_1(t)\cdot\pa_{x_2}-\du_2(t)\cdot\bb{p}_1(t)\cdot\pa_{x_2}+p \left(\bb{1}_1(t)\right)^{p-1}\bb{2}_1(t)\cdot\pa_{x_3}+\du_2(t)\left(\bb{1}_1(t)\right)^p\cdot\pa_{x_3}\ .
\end{align*}
As we see the first appearance of the $\pa_{x_3}$ direction happens in degree $p$. 

Since the initial values are $\bb{m}_i(0)=0$, for every $m=1,2,\hdots,p+1$ and $i=1,2,3$, it is easy to solve the above system:
\begin{align*}
 \bb{1}_1(t)&=\int_0^t \du_1(\tau)\dd\tau\\ \bb{1}_2(t)&=\int_0^t\du_2(\tau)\dd\tau-\int_0^t \bb{1}_1(\tau)\dd\tau= \int_0^t\du_2(\tau)\dd\tau-\int_0^t \left(\int_0^\tau\du_1(s)\dd s\right)\dd\tau\\
 \bb{2}_2(t)&=-\int_0^t\du_2(\tau)\cdot \bb{1}_1(\tau)\dd\tau =-\int_0^t\du_2(\tau)\cdot \left(\int_0^\tau\du_1(s)\dd s\right)\dd\tau\\
 \bb{p}_3(t)&=\int_0^t \left(\bb{1}_1(\tau)\right)^p\dd\tau=\int_0^t \left(\int_0^\tau\du_1(s)\dd s\right)^p\dd\tau\\
 \bb{p+1}_3(t)&=\int_0^t \du_2(
\tau)\cdot\left(\bb{1}_1(\tau)\right)^p\dd\tau=\int_0^t \du_2(\tau)\cdot\left(\int_0^\tau\du_1(s)\dd s\right)^p\dd\tau\ ,
\end{align*}
with all other components $\bb{m}_i(t)$, for $m=1,2,\hdots,p+1$, equal to zero. 
\end{example}
\section{Addapted coordinates}\label{sec:coordinates}

\paragraph{The idea of adapted coordinates}
Recall the expansion
$$\End^t(u+s\cdot\du)\overset{loc}=q(t)+s\cdot \bb{1}(t,\du)+{s^2}\cdot \bb{2}(t,\du)+{s^3}\cdot \bb{3}(t,\du)+\hdots$$ (with $\bb{m}(t,\du)$'s as in Lemma~\ref{lem:b_k}) which describes the derivatives of the end-point maps $\End^t$ at $u$ in  the direction  $\du$. As we have seen, the formulas for $\bb{m}(t,\du)$'s are quite complicated. Our goal is to introduce, for each  $m=1,2,\hdots$, a  family of $\R^n$-valued maps $\q{m}(t,\du)=\left(\q{m}_a(t,\du)\right)_{a=1,2,\hdots,n}$,  which satisfy the following two properties:
\begin{enumerate}[(i)]
    \item there is 1-1 correspondence between $(\bb{1}(t,\du),\hdots,\bb{k}(t,\du))$ and $\left(\q{1}(t,\du),\hdots,\q{k}(t,\du)\right)$ for each natural $k$;
    \item \label{point:ii} the derivatives $\dotq{m}(t,\du)$ are linear in $\du$. 
\end{enumerate}
It is therefore justified to treat $\left(\q{1},\hdots,\q{k}\right)\in\R^{n\cdot k}$ as coordinates on the space of $k$-jets $\T^k_{q(t)}M\ni (\bb{1},\hdots,\bb{k})$'s. We will  call them \emph{adapted coordinates}, meaning that they are specially adapted to the geometry of the control system \eqref{eqn:lin_ctrl_syst} along the trajectory $q(t)$.
\smallskip 

The idea is to define $\q{k}(t,\du)$'s comes from Lemma~\ref{lem:composition} -- we apply formulas \eqref{eqn:composition} describing the expansion of a composition of a function and a curve to $s\mapsto q_s(t)$, with $\bb{m}(t,\du)$'s playing the roles of $\gamma^{(m)}$'s and some, a priori unknown, $m$-linear symmetric forms $\lin{m}_a(t):\underbrace{\T_{q(t)}M\times\hdots\times\T_{q(t)}M}_{m}\lra\R$ playing the roles of the derivatives $\D^m_{\gamma(0)}f$'s. The evolution of these $\lin{m}_a(t)$ will be determined later to guarantee that the evolution of all $\q{m}(t,\du)$'s satisfies condition $\eqref{point:ii}$.  The geometric meaning of this construction will be discussed in  Lemma~\ref{lem:adapted_geometric}. 

\paragraph{Adapted coordinates in degrees 1, 2, and 3} Let us study a few examples of a low degree before passing to the general case. For each $a=1,2\hdots,n$ define:
\begin{align}
\label{eqn:q_1}
\q{1}_a(t,\du):=&\lin{1}_a(t)\left[\bb{1}(t,\du)\right]\\
\label{eqn:q_2}
\q{2}_a(t,\du):=&\lin{1}_a(t)\left[\bb{2}(t,\du)\right]+\frac 1{2!}\lin{2}_a(t)\left[\bb{1}(t,\du),\bb{1}(t,\du)\right]\\
\label{eqn:q_3}
\q{3}_{a}(t,\du):=&\lin{1}_a(t)[\bb{3}(t,\du)]+\lin{2}_a(t)\left[\bb{2}(t,\du),\bb{1}(t,\du)\right]+\\
\notag&\frac{1}{3!}\lin{3}_a(t)\left[\bb{1}(t,\du),\bb{1}(t,\du),\bb{1}(t,\du)\right]\ ,
\end{align}
 We postulate $\{\lin{1}_a(0)\ |\ a=1,2,\hdots,n\}$ to be a fixed basis of $\T^\ast_{q(0)}M$, and set  $\lin{k}_a(0)=0$ for $k=2$, $k=3$, and each $a=1,2,\hdots,n$. Now the evolution of $\q{1}_a(t)$ reads as
\begin{align*}
\dotq{1}_a(t,\du)=&\dotlin{1}_a\left[\bb{1}\right]+\lin{1}_a\left[\dotbb{1}\right]=\underbrace{\dotlin{1}_a\left[\bb{1}\right]+\lin{1}_a\left[\sum_i u_i\D X_i\left[ \bb{1}\right]\right]}_{\text{assume $=0$ to satisfy \eqref{point:ii}}}+\lin{1}_a\left[\sum_i \du_i X_i\right]=\\
&\sum_i \du_i \lin{1}_a[X_i]\ .
\end{align*}
Therefore we postulate the following evolution of $\lin{1}_a(t)$:
\begin{equation}
\label{eqn:ew_Phi_1}
\dotlin{1}_a(t)[b]+\sum_i u_i\cdot\lin{1}_a(t)\left[\D X_i[b]\right]=0\quad\text{ for every $b\in \T_{q(t)}M$.}
\end{equation}
Clearly, the above evolution is dual to the evolution on $\T_{q(t)}M$ by the flow of the time-dependent vector field $q\longmapsto \sum_i u_i(t)X_i(q)$. This guarantees that  for each $t\in[0,T]$ covectors $\lin{1}_a(t)$ form a basis of $\T^\ast_{q(t)} M$. In consequence, the procedure of constructing maps  $\q{k}(t)$ from   $\bb{k}(t)$'s will be reversible for each $k=1,2,\hdots$. 
\medskip

In degree two we have:
\begin{align*}
\dotq{2}_{a}(t,\du)=&\dotlin{1}_a\left[\bb{2}\right]+\lin{1}_a\left[\dotbb{2}\right]+\frac 1{2!}\dotlin{2}_a\left[\bb{1},\bb{1}\right]+\lin{2}_a\left[\dotbb{1},\bb{1}\right]=\\
&\underbrace{\dotlin{1}_a\left[\bb{2}\right]+
\sum_i u_i\cdot\lin{1}_a\left[\D X_i\left[\bb{2}\right]\right]}_{=0\text{ by \eqref{eqn:ew_Phi_1}}}+\\
&\underbrace{\frac 1{2!}\dotlin{2}_a\left[\bb{1},\bb{1}\right]+\sum_i u_i\cdot\left\{\frac 1{2!}\lin{1}_a\left[\D^2X_i\left[\bb{1},\bb{1}\right]\right] +\lin{2}_a\left[\D X_i\left[\bb{1}\right],\bb{1}\right]\right\}}_{\text{assume $=0$ to satisfy \eqref{point:ii}}}+\\
&\sum_i \du_i\left\{\lin{1}_a\left[\D  X_i[\bb{1}]\right]+\lin{2}_a\left[X_i, \bb{1}\right]\right\}\ .
\end{align*}

Therefore we postulate 
$$
\dotlin{2}_a[b,b]+\sum_i u_i\cdot\left\{ 2\lin{2}_a\left[\D X_i[b],b\right]+\lin{1}_a\left[\D^2 X_i[b,b]\right]\right\} =0\qquad\text{ for every $b\in \T_{q(t)}M$.}
$$
If the above holds then, since $\lin{2}_a$ is symmetric 2-linear,  for every $b,b'\in \T_{q(t)}M$ we have:
\begin{equation}
\label{eqn:ew_Phi_2}
\dotlin{2}_a[b,b']+\sum_i u_i\cdot\left\{\lin{1}_a\left[\D^2 X_i[ b,b']\right]+\lin{2}_a\left[\D X_i[b],b'\right]+\lin{2}_a\left[\D X_i[b'],b\right]\right\}=0\ .
\end{equation}

Analogously in degree 3: 
\begin{align*}
\dotq{3}_{a}(t)&=\dotlin{1}_a\left[\bb{3}\right]+\lin{1}_a\left[\dotbb{3}\right]+\dotlin{2}_a\left[\bb{2},\bb{1}\right]+\lin{2}_a\left[\dotbb{2},\bb{1}\right]+\lin{2}_a\left[\bb{2},\dotbb{1}\right]+\frac 1{3!}\dotlin{3}_a\left[\bb{1},\bb{1},\bb{1}\right]+\frac 1{2!}\lin{3}_a\left[\dotbb{1},\bb{1},\bb{1}\right] =\\
&\underbrace{\dotlin{1}_a\left[\bb{3}\right]+\sum_i u_i\cdot
\lin{1}_a\left[\D X_i[\bb{3}]\right]}_{=0\text{ by \eqref{eqn:ew_Phi_1}}}+\\
&\underbrace{\dotlin{2}_a\left[\bb{2},\bb{1}\right]+\sum_i u_i\cdot\left\{\lin{1}_a\left[\D^2 X_i[\bb{2},\bb{1}]\right]+\lin{2}_a\left[\D X_i[\bb{2}],\bb{1}\right]+\lin{2}_a\left[\D X_i[\bb{1}],\bb{2}\right]\right\}}_{\text{$=0$ by \eqref{eqn:ew_Phi_2}}}+\\
&\underbrace{\frac 1{3!}\dotlin{3}_a\left[\bb{1},\bb{1},\bb{1}\right]+\sum_i u_i \left\{\frac 1{3!}\lin{1}_a\left[\D^3 X_i[\bb{1},\bb{1},\bb{1}]\right]+\frac 1{2!}\lin{2}_a\left[\D^2 X_i[\bb{1},\bb{1}],\bb{1}\right]+\frac 1{2!}\lin{3}_a\left[\D X_i[\bb{1}],\bb{1},\bb{1}\right]\right\}}_{\text{assume $=0$ to satisfy \eqref{point:ii}}}\\
&\sum_i \du_i\left\{\lin{1}_a\left[\D X_i[\bb{2}]\right]+\lin{1}_a\left[\D^2X_i[\bb{1},\bb{1}]\right]+\lin{2}_a\left[\D X_i[\bb{1}], \bb{1}\right]+\lin{2}_a\left[X_i, \bb{2}\right]+\frac 1{2!}\lin{3}_a\left[X_i,\bb{1},\bb{1}\right]\right\}\ .
\end{align*}

Therefore we postulate that for every $b\in \T_{q(t)}M$
\begin{equation}
    \label{eqn:ew_Phi_3}
    \dotlin{3}_a[b,b,b]+\sum_i u_i\cdot\left\{3\lin{3}_a\left[ \D X_i[b],b,b\right]+3\lin{2}_a\left[\D^2 X_i[b,b],b\right]  +\lin{1}_a\left[\D^3 X_i[b,b,b]\right]\right\}=0 
\end{equation}

Summing up, under assumptions \eqref{eqn:ew_Phi_1}--\eqref{eqn:ew_Phi_3}, our new coordinates \eqref{eqn:q_1}--\eqref{eqn:q_3},  evolve as follows: 
\begin{align*}
\dotq{1}_a=&\sum_i \du_i \Phi^{(1)}_a\left[X_i\right]\\
\dotq{2}_a=&\sum_i \du_i\left\{ \Phi^{(1)}_a\left[\D X_i[\bb{1}]\right]+\Phi^{(2)}_a\left[X_i, \bb{1}\right]\right\} \\
\dotq{3}_a=&\sum_i \du_i\left\{\Phi^{(1)}_a\left[\D X_i[\bb{2}]\right] +\Phi^{(1)}_a\left[\D^2 X_i[\bb{1},\bb{1}]\right]+\Phi^{(2)}_a\left[\D X_i[\bb{1}], \bb{1}\right]+\Phi^{(2)}_a\left[X_i, \bb{2}\right]+\frac 1{2!}\Phi^{(3)}_a\left[X_i,\bb{1},\bb{1}\right]\right\}
\end{align*}

\paragraph{The general case}  The pattern observed for terms up to order 3, continues in all degrees. We may thus generalize the transformations \eqref{eqn:q_1}--\eqref{eqn:q_3} (satisfying conditions \eqref{eqn:ew_Phi_1}--\eqref{eqn:ew_Phi_3}) in the following definition.

\begin{df} Consider a trajectory $q(t)$ of a control-linear system \eqref{eqn:lin_ctrl_syst} corresponding to the control $u\in \U$. Choose a local coordinate system on $M$ and for each $m=1,2,\hdots$ define a time-dependent family of symmetric $m$-linear maps
 $$\lin{m}(t)=\left(\lin{m}_a(t)\right)_{a=1,2,\hdots,n}:\underbrace{\T_{q(t)}M\times\hdots\times\T_{q(t)}M}_{\text{$m$ times}}\lra \R^n$$ 
 by setting the following conditions:
\begin{itemize}
    \item we set $\{\lin{1}_a(0)\}_{a=1,2\hdots,n}$ to be a basis of $\T^\ast_{q(0)}M$, while for $m=2,3,\hdots$ maps $\lin{m}(0)$ can be arbitrary, 
    \item $\lin{m}(t)$'s are subject to the following ODEs:
 \begin{equation}
 \label{eqn:ew_Phi}
     \dotlin{m}(t)[\underbrace{b,....,b}_{\text{$m$ times}}] + \sum_i u_i(t)\cdot\left\{\sum_{s=1}^m {m \choose s}\lin{m-s+1}(t)\left[\D^s_{q(t)}X_i[\underbrace{b,\hdots, b}_{\text{$s$ times}}], \underbrace{b,\hdots,b}_{\text{$m-s$ times}}\right]\right\}=0\ .
 \end{equation}
\end{itemize}

For each $k=1,2,\hdots$ and $t\in[0,T]$ the assignment
$$\AddCoord{k}(t): \T^k_{q(t)}M\ni(\bb{1},\hdots,\bb{k})\longmapsto(\q{1},\hdots,\q{k})\in \R^{n\cdot k}\ , $$
given by (within the notation convention of Rem.~\ref{multi_ind_not})
\begin{equation}
\label{eqn:def_q_k}
    \q{m}=\left(\q{m}_a\right)_{a=1,2,\hdots,n}:=\sum_{\alpha,\ w(\alpha)=m} \frac 1{\alpha!} \cdot \lin{|\alpha|}(t) \left[\vec{b}^{\alpha}\right]\ 
    \end{equation}
will be called the \emph{transformation of adapted coordinates} of degree $k$ at time $t$.
\end{df}

We summarize basic properties of the above notion in the following result.

\begin{thm}\label{thm:adapted}
For every $t\in[0,T]$, and every $k$, the transformation of adapted coordinates $\AddCoord{k}(t)$ does not depend on the choice of a local coordinate system. Further it is an isomorphism between the fibre $\T^k_{q(t)}M\ni (\bb{1},\hdots,\bb{k})$ and $\R^{n\cdot k}\ni (\q{1},\hdots,\q{k})$ respecting the natural graded structures in which $\deg(\bb{i})=\deg(\q{i})=i$.
\smallskip

For any control $\du\in\T_u\U\simeq \U$ consider the $k$-jet of the curve $s\mapsto \End^t[u+s\cdot\du]$ at $s=0$  given in local coordinates by 
$$\End^t(u+s\cdot\du)\overset{loc}=q(t)+s\cdot \bb{1}(t,\du)+s^2\cdot\bb{2}(t,\du)+\hdots s^k\cdot\bb{k}(t,\du)
+o(s^k)\ .$$
Let $(\q{1}(t,\du),\hdots,\q{k}(t,\du))$ be the image of the $k$-jet $(\bb{1}(t,\du),\hdots,\bb{k}(t,\du))$ under the transformation of adapted coordinates $\AddCoord{k}(t)$. That is,
for $m=1,2,\hdots,k$ we have 
\begin{equation}
\label{eqn:def_q_k_particular}
    \q{m}(t,\du):=\sum_{\alpha,\ w(\alpha)=m} \frac 1{\alpha!} \cdot \lin{|\alpha|}(t) [\vec{b}(t,\du)^{\alpha}]\ .
    \end{equation}  
    Then, curves $\q{m}(t,\du)$ satisfy the following system of ODEs:
\begin{equation}
\label{eqn:dot_q_k}
    \dotq{m}(t,\du)=\sum_{i} \du_i(t)\cdot\left\{\sum_{\alpha,\ \beta,\ w(\alpha)+w(\beta)=m-1} \frac{1}{\alpha!} \cdot  \frac 1{\beta !} \cdot  \lin{|\alpha|+1}(t) \left[\D^{|\beta|}_{q(t)}X_i[\vec{b}(t,\du)^\beta]
,\vec{b}(t,\du)^{\alpha}\right]\right\}
\end{equation}
\end{thm}

\begin{proof}
The fact that the map $\AddCoord{k}(t)$ is well-defined will be proved in the next paragraph (Lemma~\ref{lem:adapted_geometric}), where we shall give a geometric interpretation of this construction.

Let us address the problem of reversibility. Note that for each $t\in[0,T]$ the family $\{\lin{1}_a(t)\}_{a=1,\hdots,n}$ is a basis of $\T^\ast_{q(t)}M$. Indeed this is easily seen from $\eqref{eqn:ew_Phi_1}$, which easily implies that 
$$\lin{1}_a(t)[\bb{1}(t)]=\operatorname{const}$$
whenever $\bb{1}(t)$ is subject to 
$\dotbb{1}(t)=\sum_i u_i(t)\cdot \D_{q(t)}X_i[\bb{1}(t)]$. The latter is the evolution in $\T M$ induced by the flow of the (time-dependent) 
control vector field $q\mapsto \sum_i u_i(t) X_i(q)$. As this flow consists of diffeomorphisms (see \cite[Sec. 2]{MJ_WR_contact_pmp}), then the evolution of $\bb{1}(t)$ is reversible, and hence so is the evolution of $\lin{1}(t):=\left(\lin{1}_a(t)\right)_{a=1,2\hdots,n}$.

We conclude that the assignment 
$$\T_{q(t)}M\ni \bb{1}\longmapsto \left(\q{1}=\lin{1}(t)[\bb{1}]\right)\in \R^n$$
is reversible for each $t\in[0,T]$. Similarly, so is 
$$\T^2_{q(t)}M\ni (\bb{1},\bb{2})\longmapsto \left(\q{1},\q{2}=\lin{1}(t)[\bb{2}]+\lin{2}(t)[\bb{1},\bb{1}]\right)\in \R^{2n}\ , $$
as we may express $\bb{1}$ in terms of $\q{1}$ and then reverse $\lin{1}(t)$ to get $\bb{2}$ from $\q{2}$. 

In higher degrees an analogous argument proves that formulas \eqref{eqn:def_q_k} define isomorphisms between the space of $k$-jets of curves in $M$ at $q(t)$ and the space $\R^{k\cdot n}$ 
$$\T^k_{q(t)}M\ni (\bb{1},\hdots,\bb{k})\longmapsto(\q{1},\hdots,\q{k})\in \R^{k\cdot n} \ .$$
Finally note that formulas \eqref{eqn:def_q_k} respect the gradings given by $\deg(\bb{m})=\deg(\q{m})=m$.
\medskip

Now let us prove the remaining part of the assertion.  Differentiation of \eqref{eqn:def_q_k_particular} gives us 
\begin{align*}
    \dotq{m}(t,\du)=\sum_{\alpha,\ w(\alpha)=m} \dotlin{|\alpha|} [\vec{b}(t,\du)^{\alpha}]+\sum_{\alpha,\ w(\alpha)=m}\sum_{l\leq m} \frac{1}{\alpha!} \cdot a_l\cdot \lin{|\alpha|} [\dotbb{l},\vec{b}(t,\du)^{\alpha-1_l}]\ , 
\end{align*}
where naturally $\alpha-1_l=(a_1,\hdots,a_{l-1},a_l-1,a_{l+1},\hdots,a_k)$. Now note that by \eqref{eqn:ew_Phi} the derivatives $\dotlin{|\alpha|}_a$ are linear in $u_i$'s, while by \eqref{eqn:bb_k} the derivatives $\dotbb{l}$ have a part linear in $u_i$'s and a part linear in $\du_i$'s. Therefore, $\dotq{m}$ splits into a part linear in $u_i$'s and a part linear in $\du_i$'s, i.e
$$\dotq{m}(t,\du)=\sum_i u_i\cdot A_{i} +\sum_i \du_i\cdot B_{i} \ ,$$
where $A_{i}$ and $B_{i}$ do not depend on neither $u_i$'s, nor $\du_i$'s.
We will  now calculate these two parts separately, proving that the first one is zero, while the second one is precisely the right-hand side of \eqref{eqn:dot_q_k}. This will end the proof. 
\medskip

Let us begin with the second part. Since $\dotlin{m}(t)$ do not depend on $\du_i$'s, only the derivatives $\dotbb{l}$ contribute giving
\begin{align*}
    \sum_i\du_i\cdot B_{i}=  &\sum_{\alpha,\ w(\alpha)=m}\sum_{l\leq m} \frac{1}{\alpha!} \cdot a_l\cdot \lin{|\alpha|} [\text{part of $\dotbb{l}$ linear in $\du_i$'s},\vec{b}(t)^{\alpha-1_l}]\overset{\eqref{eqn:bb_k}}=\\
    &\sum_{\alpha,\ w(\alpha)=m}\sum_{l\leq m} \frac{1}{\alpha!} \cdot a_l\cdot \lin{|\alpha|} \left[\sum_{i} \du_i\cdot\left\{\sum_{\beta,\  w(\beta)=l-1} \frac 1{\beta !} \cdot\D^{|\beta|}_{q(t)}X_i[\vec{b}(t)^\beta]\right\}
    ,b(t)^{\alpha-1_l}\right]=\\
    &\sum_{i} \du_i\cdot\left\{\sum_{\alpha,\ w(\alpha)=m}\sum_{l\leq m} \frac{1}{(\alpha-1_l)!} \cdot \sum_{\beta,\  w(\beta)=l-1} \frac 1{\beta !} \cdot  \lin{|\alpha-1_l|+1} \left[\D^{|\beta|}_{q(t)}X_i[\vec{b}(t)^\beta]
    ,\vec{b}(t)^{\alpha-1_l}\right]\right\}
\end{align*}
Now observe that a triple $(\alpha,l,\beta)$ where $w(\alpha)=m$, $a_l>0$ and $w(\beta)=l-1$ uniquely determines a pair of multi-indexes $(\gamma=\alpha-1_l,\beta)$ satisfying $w(\gamma)+w(\beta)=(m-l)+l-1=m-1$. Thus we may change the summation order in the expression above to obtain 
$$\sum_i \du_i\cdot B_{i}=\sum_{i} \du_i\cdot\left\{\sum_{\gamma,\ \beta,\ w(\gamma)+w(\beta)=m-1} \frac{1}{\gamma!} \cdot  \frac 1{\beta !} \cdot  \lin{|\gamma|+1} [\D^{|\beta|}_{q(t)}X_i[\vec{b}(t)^\beta]
,\vec{b}(t)^{\gamma}]\right\}$$
in agreement with \eqref{eqn:dot_q_k}.
\medskip

Now let us calculate the part of $\dotq{m}(t,\du)$ linear in $u_i$'s. Proceeding as in the previous part of the proof we arrive at
\begin{align*}
    \sum_i u_i&\cdot A_{i}=\sum_{\alpha,\ w(\alpha)=m}\dotlin{|\alpha|} [\vec{b}(t,\du)^{\alpha}]+\sum_{\alpha,\ w(\alpha)=m}\sum_{l\leq m} \frac{1}{\alpha!} \cdot a_l\cdot \lin{|\alpha|} [\text{part of $\dotbb{l}$ linear in $u_i$'s},\vec{b}(t,\du)^{\alpha-1_l}]\overset{\eqref{eqn:bb_k}}=\\
    &\sum_{\alpha,\ w(\alpha)=m} \dotlin{|\alpha|} [\vec{b}(t)^{\alpha}]+\sum_{\alpha,\ w(\alpha)=m}\sum_{l\leq m} \frac{1}{\alpha!} \cdot a_l\cdot \lin{|\alpha|} \left[\sum_{i} u_i\cdot\left\{\sum_{\beta,\  w(\beta)=l} \frac 1{\beta !} \cdot\D^{|\beta|}_{q(t)}X_i[\vec{b}(t)^\beta]\right\}
    ,b(t)^{\alpha-1_l}\right]=\\
    &\sum_{\alpha,\ w(\alpha)=m} \dotlin{|\alpha|} [\vec{b}(t)^{\alpha}]+\sum_{i} u_i\cdot\left\{\sum_{\alpha,\ w(\alpha)=m}\sum_{l\leq m} \frac{1}{(\alpha-1_l)!} \cdot \sum_{\beta,\  w(\beta)=l} \frac 1{\beta !} \cdot  \lin{|\alpha-1_l|+1} \left[\D^{|\beta|}_{q(t)}X_i[\vec{b}(t)^\beta]
    ,\vec{b}(t)^{\alpha-1_l}\right]\right\}=\\
    &\sum_{\alpha,\ w(\alpha)=m} \dotlin{|\alpha|} [\vec{b}(t)^{\alpha}]+\sum_{i} u_i\cdot\left\{\sum_{\gamma,\ \beta,\ w(\gamma)+w(\beta)=m} \frac{1}{\gamma!} \cdot  \frac 1{\beta !} \cdot  \lin{|\gamma|+1} [\D^{|\beta|}_{q(t)}X_i[\vec{b}(t)^\beta]
,\vec{b}(t)^{\gamma}]\right\},
\end{align*}
where the changes of summation order in the last passage are made analogously as before.
We would like to show that the above equals to zero. To see this note that formula \eqref{eqn:ew_Phi} describes the evolution of a $m$-linear map $\lin{m}$. Due to the symmetry of $\lin{m}$ we also have
$$ - \dotlin{m}[v_1,....,v_m] =\sum_i u_i\cdot \left\{\sum_{\sigma \in \Sigma_m} \sum_{s=1}^m \frac{1}{s!(m-s)!}\lin{m-s+1}\left[ \D^s_{q(t)}X_i[v_{\sigma(1)},\hdots, v_{\sigma(s)}], v_{\sigma(s+1)},\hdots,v_{\sigma(m)}\right]\right\}$$
for every $m$-tuple of vectors $v_1,v_2\hdots,v_m\in\T_{q(t)}M$. Hence
\begin{align*}
    \sum_{\alpha,\ w(\alpha)=m}& \frac{1}{\alpha!} \cdot \dotlin{|\alpha|} [\vec{b}(t)^{\alpha}]=\\
    &-\sum_i u_i \sum_{\alpha,\ w(\alpha)=m} \frac{1}{\alpha!} \sum_{\sigma \in \Sigma_{|\alpha|}} \sum_{s=1}^{|\alpha|} \frac{1}{s!({|\alpha|}-s)!}\lin{{|\alpha|}-s+1}[ \D^s_{q(t)}X_i[\bb{\sigma(1)},\hdots, \bb{\sigma(s)}], \bb{\sigma(s+1)},\hdots,\bb{\sigma({|\alpha|})}]
\end{align*}
Due to the fact that $\lin{{|\alpha|}-s+1}_a$ and $ \D^s_{q(t)}X_i$ are multi-linear, we may identify two permutations that have the same $s$ initial elements. The number of such permutations for a given $\alpha$ is precisely $s!(|\alpha|-s)!$. It follows that we can identify the above triple sum over $(\alpha, \sigma,s)$ as a sum over sub-divisions of the multi-index $\alpha$ into a sum $\alpha=\beta+\gamma$ (so $s=|\beta|$, $|\alpha|-s=|\gamma|$, and $w(\alpha)=w(\beta)+w(\gamma)$). The latter are however taken with multiplicity $\binom{\alpha}{\beta}=\prod_i\binom{\alpha_i}{\beta_i}$, as $\binom{\alpha_i}{\beta_i}$ is a number of ways an ordered sequence of $\alpha_i$ elements $(b_i,\hdots,b_i)$ can be divided between in two groups of $\beta_i$ and $\gamma_i=\alpha_i-\beta_i$ elements. Since $\alpha!=\binom{\alpha}{\beta}\cdot\beta!\cdot\gamma!$. We arrive at   
$$\sum_{\alpha,\ w(\alpha)=m} \dotlin{|\alpha|} [\vec{b}(t,\du)^{\alpha}]=-\sum_{i} u_i\cdot\left\{\sum_{\gamma,\ \beta,\ w(\gamma)+w(\beta)=m} \frac{1}{\gamma!} \cdot  \frac 1{\beta !} \cdot  \lin{|\gamma|+1} [\D^{|\beta|}_{q(t)}X_i[\vec{b}(t,\du)^\beta]
,\vec{b}(t,\du)^{\gamma}]\right\}\ ,$$
proving that $\sum_i u_i\cdot A_{i}=0$. This ends the proof.
\end{proof}

\begin{example}[Generalized Martinet system, part 2.]\label{ex:2}
 Let us apply the above theory to Example~\ref{ex:1}. Equation \eqref{eqn:ew_Phi_1} reads as
$$\dotlin{1}(t)[\cdot]=-\lin{1}(t)\left[\D_{q(t)}X_2[\cdot]\right]=\<\dd x_1,\cdot>\cdot\lin{1}(t)[\pa_{x_2}]\ .$$
For a natural choice of initial value $\lin{1}(0)=\operatorname{Id}_{\R^3}$ this implies that $\lin{1}(t)$ is the following linear isomorphism on $\R^3$
$$\lin{1}(t)[\pa_{x_1}]=\pa_{x_1}+t\cdot\pa_{x_2},\quad  \lin{1}(t)[\pa_{x_2}]=\pa_{x_2}\quad \text{and}\quad\lin{1}(t)[\pa_{x_3}]=\pa_{x_3} . $$ 
In higher degrees $\lin{m}(t)\equiv0$, for $m=2,3,\hdots,p-1,p+1,\hdots$,  with the sole exception of a single component of $\lin{p}(t)$, namely 
$$\lin{p}(t)[\underbrace{\pa_{x_1},\hdots,\pa_{x_1}}_p]=-p!\cdot t\cdot\pa_{x_3}\ .$$ 

We can now use formula \eqref{eqn:def_q_k} definining the transformation of adapted coordinates to arrive at 
\begin{align*}
    \q{1}(t)&=\lin{1}(t)\left[\bb{1}(t)\right]\\
     \q{2}(t)&=\lin{1}(t)\left[\bb{2}(t)\right]+\lin{2}(t)\left[\bb{1}(t),\bb{1}(t)\right]=\lin{1}(t)\left[\bb{2}(t)\right]\\
     \hdots\\
     \q{m}(t)&=\lin{1}(t)\left[\bb{m}(t)\right]\\
     \hdots\\
     \q{p}(t)&=\lin{1}(t)\left[\bb{p}(t)\right]+\frac 1{p!}\lin{p}(t)\left[\bb{1}(t)\hdots,\bb{1}(t)\right]=\lin{1}(t)\left[\bb{p}(t)\right]-{t\left(\bb{1}_1(t)\right)^p\cdot\pa_{x_3}}\\
     \q{p+1}(t)&=\lin{1}(t)\left[\bb{p+1}(t)\right]+\frac 1{(p-1)!}\lin{p}(t)\left[\bb{1}(t)\hdots,\bb{1}(t),\bb{2}(t)\right]=\\
     &\phantom{==} \lin{1}(t)\left[\bb{p+1}(t)\right]-p\cdot t\left(\bb{1}_1(t)\right)^{p-1}\bb{2}_1(t)\cdot\pa_{x_3}
\end{align*}

We can differentiate the above formulas, or use evolution equations \eqref{eqn:dot_q_k} to get evolution equations for $\q{m}$'s which have a bit simpler form then those for $\bb{m}$'s
\begin{align*}
    \dotq{1}(t)&=\du_1(t) \cdot(\pa_{x_1}+t\cdot \pa_{x_2})+\du_2(t)\cdot\pa_{x_2}\\
    \dotq{2}(t)&=-\du_2(t)\cdot\bb{1}_1(t)\cdot\pa_{x_2}\\
    \hdots\\
    \dotq{m}(t)&=-\du_2(t)
    \cdot \bb{m-1}_1(t)\cdot \pa_{x_2}\\
    \hdots\\
    \dotq{p}(t)&=-\du_2(t)\cdot \bb{p-1}_1(t)\cdot\pa_{x_2}-\du_1(t)\cdot p\cdot t \left(\bb{1}_1(t)\right)^{p-1}\cdot\pa_{x_3}\\
    \dotq{p+1}(t)&=-\du_2(t)\cdot \bb{p}_1(t)\cdot\pa_{x_2}+\du_2(t)\left(\bb{1}_1\right)^p\cdot\pa_{x_3}+\du_1(t)\cdot p(p-1)\cdot t \left(\bb{1}_1(t)\right)^{p-2}\bb{2}_1(t)\cdot\pa_{x_3}\ .
\end{align*}
The solutions are
\begin{align*}
 \q{1}_1(t)&=\bb{1}_1(t)=\int_0^t \du_1(\tau)\dd\tau\\ \q{1}_2(t)&=\int_0^t\du_2(\tau)+\int_0^t\tau\cdot\du_1(\tau)\dd\tau\\
 \q{2}_2(t)&=-\int_0^t\du_2(\tau)\cdot \bb{1}_1(\tau)\dd\tau =-\int_0^t\du_2(\tau)\cdot \left(\int_0^\tau\du_1(s)\dd s\right)\dd\tau\\
 \q{p}_3(t)&=-t\left(\bb{1}_1(t)\right)^p=-t\cdot \left(\int_0^t\du_1(\tau)\dd \tau\right)^p\\
\q{p+1}_3(t)&=\int_0^t\du_2(\tau)\left(\bb{1}_1(\tau)\right)^p\dd\tau=\int_0^t\du_2(\tau)\cdot \left(\int_0^\tau\du_1(s)\dd s\right)^p\dd\tau\ ,
\end{align*}
with all other components $\q{m}_i(t)$, for $m=1,2,\hdots,p+1$, equal to zero.
\end{example}

\paragraph{Geometric interpretation of the adapted coordinates}
So far the construction of adapted coordinates may be seen as a computational trick which helps to simplify the evolution equations for curves $\bb{m}(t,\du)$. However, as we shall see in this paragraph, it has a natural geometric interpretation. 
\medskip

Consider $q(t)$ a trajectory of the control-linear system \eqref{eqn:lin_ctrl_syst} corresponding to the control $u\in\U$ and initial condition $q(0)=q_0$. 
Let now $q_s(t)$  be a family of solutions of \eqref{eqn:lin_ctrl_syst} corresponding to the same control $u$ and a smooth curve of initial conditions $s\mapsto q_s(0)$, passing through $q_0=q_0(0)$ (so that $q_{s=0}(t)=q(t)$).  It turns out that the assignment $q_s(0)\longmapsto q_s(t)$ gives rise to a well-defined map on  $k$-jets at $s=0$:\footnote{This is a simple consequence of the theorem about the regularity of the dependence of a solution of an ODE (in the sense of Caratheodory) on the initial condition.}
$$\T^k_{q_0}M\ni[q_s(0)]_{\sim k}\longmapsto[q_s(t)]_{\sim k}\in \T^k_{q(t)}M\ . $$
We shall denote this map -- the $k$-th tangent lift of the flow of the control vector field $X_u$ --  by $\T^k X^t_u$. It describes the natural action on $k$-jets of the evolution (flow) of the control system \eqref{eqn:lin_ctrl_syst} for a fixed control $u$. The following results explain the relation of $\T^k X^t_u$ with the transformation of adapted coordinates. 
 
\begin{lem}\label{lem:adapted_geometric} Let $q(t)$ be a trajectory of the control-linear system \eqref{eqn:lin_ctrl_syst} corresponding to a control $u\in \U$. Consider a 1-parameter family of $k$-jets $\vec{b}(t)=(\bb{1}(t),\hdots,\bb{k}(t))\in \T^k_{q(t)}M$. Then 
$$\vec{b}(t)=\T^kX_u^t[\vec{b}(0)]\quad \text{for every $t\in[0,T]$ if and only if}$$
 $\vec{b}(t)=(\bb{1}(t),\hdots,\bb{k}(t))$ is constant under the transformation of adapted coordinates $\AddCoord{k}(t)$. 
\smallskip

In other words 
$$\left(\AddCoord{k}(t)\right)^{-1}[(\q{1},\hdots,\q{k})]=\vec{b}(t)=\T^k X_u^t[\vec{b}(0)]\ ,$$
where $(\q{1},\hdots,\q{k})=\AddCoord{k}(0)[\vec{b}(0)]$. In particular, the construction of the transformation of adapted coordinates does not depend on the choice of the local coordinate system on $M$. 
\end{lem}
\begin{proof}
Choose a local coordinate system on $M$ and let $q_s(t)$ be a family of solutions of \eqref{eqn:lin_ctrl_syst} as described at the beginning of this paragraph. Now 
$$\dot{q_s}(t)=\sum_iu_i(t)X_i(q_s(t))\quad\text{with the initial condition $q_s(0)$,}$$
hence $\vec{b}(t)=(\bb{1}(t),\hdots,\bb{k}(t))$ -- the $k$-jet of $s\mapsto q_s(t)$ at $s=0$, which by definition equals to $\T^kX_u^t[\vec{b}(0)]$ -- satisfies an ODE obtained by a repetitive differentiation of the above equation with respect to $s$. This will, however, be a special case of an ODE obtained by an analogous differentiation of equation \eqref{eqn:q_s_t} if we take $\du=0$. As we know from Lemma~\ref{lem:b_k}, the latter equation is just \eqref{eqn:bb_k}, and so $\vec{b}(t)$ is subject to \eqref{eqn:bb_k} for $\du=0$. Now, by the results of Theorem~\ref{thm:adapted}, after an application of the transformation of adapted coordinates, the curve $\vec{q}(t)=(\q{1}(t),\hdots,\q{k}(t))=\AddCoord{k}(t)[\vec{b}(t)]$ would satisfy equation \eqref{eqn:dot_q_k} with $\du=0$. Thus $\dot{\vec{q}}(t)=0$ which ends the proof. 
\end{proof}

\paragraph{Getting rid of $\bb{m}$'s}
Our initial motivation behind the construction of adapted coordinates was to simplify the evolution equations for the polynomial expansion of $s\mapsto \End^t[u+s\cdot\du]$. As a result, we got the formula \eqref{eqn:dot_q_k}. Note, however, that as $\q{m}(t,\du)$'s and $\bb{m}(t,\du)$'s are related by means of the transformation of adapted coordinates which is invertible, it is possible to express the right-hand side of \eqref{eqn:dot_q_k} in terms of $\q{m}(t,\du)$'s only. Thus we can interpret equation \eqref{eqn:dot_q_k} as a time-dependent control-linear system in variables $(\q{1},\hdots,\q{k})$ (or to view it differently a control-affine system in variables  $(t,\q{1},\hdots,\q{k})$), where $\du\in\U$ plays the role of a control. This observation is summarised as follows.
 
\begin{thm}\label{thm:adapted_q}
Consider a trajectory $q(t)$ of a control-linear system \eqref{eqn:lin_ctrl_syst} corresponding to the control $u\in \U$ and choose $k\in \mathbb{N}$.  Then there exists an affine distribution $\mathcal{D}^{(k)}=\pa_t+\operatorname{span}_{\R}\{Y_1^{(k)},\hdots,Y_l^{(k)}\}$ on $\R\times\underbrace{\R^n\times\hdots\times\R^n}_{k}\ni(t,\q{1},\hdots,\q{k})$ with the following properties: 
\begin{enumerate}[(i)]
    \item consider a natural graded space structure on $\R^{1+kn}$ by setting $\deg(t)=0$ and $\deg (\q{m}_a)=m$. Then for every $i=1,2,\hdots,l$ the field $Y_i^{(k)}$ is a homogeneous vector field of degree $-1$
    \item for every control $\du\in \T_u\U\simeq \U$ the curve $p^{(k)}(t):=(t, \q{1}(t,\du),\hdots,\q{k}(t,\du))$ -- where \\ $(\q{1}(t,\du),\hdots,\q{k}(t,\du))$ is the image of the $k$-jet of $s\mapsto \End^t[u+s\cdot\du]$ at $s=0$ under the transformation of adapted coordinates $\AddCoord{k}(t)$ considered in Theorem~\ref{thm:adapted} -- is a trajectory  of the control-affine system defined by $\mathcal{D}^{(k)}$  for the control  $\du$, that is
    \begin{equation}
    \label{eqn:aff_ctr_sys}\tag{$\Sigma^{(k)}(\du)$}
    \dot p^{(k)}(t)=\pa_t+\sum_i \du_i(t)\cdot Y_i^{(k)}(p^{(k)}(t))\qquad\text{with  $p^{(k)}(0)=(0,0,\hdots,0)$.}
\end{equation}
\end{enumerate}
Moreover, for each $k=2,3,\hdots$, the field $Y^{(k)}_i$ projects to $Y^{(k-1)}_i$ under the mapping $$\R^{1+nk}\ni(t,\q{1},\hdots,\q{k-1},\q{k})\longmapsto (t,\q{1},\hdots,\q{k-1})\in\R^{1+n(k-1)} \ .$$
\end{thm}

\begin{proof}

 Fix a natural $k$ and consider evolution equations \eqref{eqn:dot_q_k}. They are of the form
$$\dotq{m}(t,\du)=\sum_i\du_i(t)\cdot \widetilde{Y}^{(m)}_{i}(t,\vec{b}(t,\du))\ ,$$
where $m=1,2,\hdots,k$ and $\widetilde{Y}^{(m)}_{i}(t,\vec{b}(t,\du))$ is polynomial in $\vec{b}(t,\du):=(\bb{1}(t,\du),\hdots,\bb{k}(t,\du))$ of graded degree $m-1$. Now we may substitute $\vec{b}(t,\du)$ with $\vec{q}(t,\du):=(\q{1}(t,\du),\hdots,\q{k}(t,\du))$ using the inverse of the transformation of adapted coordinates $\AddCoord{k}(t)$. Note that the latter transformation intertwines the natural graded structures in the space of $\bb{m}$'s and $\q{m}$'s. Thus we get
$$\dotq{m}(t,\du)=\sum_i\du_i(t)\cdot Y_{i}^{(m)}(t,\vec{q}(t,\du))\ ,$$
where  $Y^{(m)}(t,\vec{q})=\widetilde{Y}^{(m)}_{i}(t,\AddCoord{k}(t)^{-1}[\vec{q}])$ is polynomial in $\vec{q}$ of graded degree $m-1$. \end{proof}

Let us see how this result looks in low degrees. Consider the linear isomorphism  $\lin{1}(t):\T_{q(t)}M\ra\R^n$ and denote its inverse by $A(t):\R^n\ra \T_{q(t)}M$.
Then 
\begin{align*}
\bb{1}(t,\du)=A(t)[\q{1}(t,\du)]\quad\text{and}\quad \bb{2}(t,\du)=A(t)[\q{2}(t,\du)]-\frac 1{2!}A(t)\lin{2}[A\q{1}(t,\du),A\q{1}(t,\du)]\ ,
\end{align*}
and hence equations \eqref{eqn:aff_ctr_sys} look as follows
\begin{align*}
\dotq{1}=&\sum_i \du_i\cdot  \lin{1}[X_i]\\
\dotq{2}=&\sum_i \du_i\cdot \left\{\lin{1}\left[\D X_i[A \q{1}]\right]+\lin{2}\left[X_i, Aq^{(1)}\right]\right\} \\
\dotq{3}=&\sum_i \du_i\cdot \left\{\lin{1}\left[\D X_i[A\q{2}]\right]-\frac{1}{2!}\lin{1}\left[A\lin{2}[A\q{1},A\q{1}\right]+\frac{1}{2!}\lin{1}\left[\D^2X_i[A\q{1},A\q{1}\right]+\right.\\
&\lin{2}\left[\D X_i[A\q{1}], A\q{1}\right]+\lin{2}\left[X_i, A\q{2}\right]-\\
&\left.\frac{1}{2!}\lin{2}\left[X_i,A\lin{2}[A\q{1},A\q{1}]\right]+\frac1{2!}\lin{3}\left[X_i,A\q{1},A\q{1}\right]\right\}
\end{align*}
Unfortunately, we were unable to derive exact formulas for $\dotq{m}$'s in every degree, as they become very complicated with the increasing $m$.  

\begin{example}[Generalized Martinet system, part 3.]\label{ex:3} In the previous Example~\ref{ex:2} we derived adapted coordinates and their evolution equations for a generalized Martinet system. In that particular situation it is easy to reverse the transformation of adapted coordinates to arrive at
\begin{align*}
    \bb{m}_1(t)&=\q{m}_1\qquad \text{for $m=1,2,\hdots, p, p+1$;}\\
    \bb{m}_2(t)&=\q{m}_2(t)-t\cdot \q{m}_1(t)\qquad \text{for $m=1,2,\hdots, p,p+1$; and}\\
    \bb{m}_3(t)&=\q{m}_3(t)\qquad \text{for $m=1,2,\hdots,p-1$; and}\\
\bb{p}_3(t)&=\q{p}_3(t)+t\left(\q{1}_1(t)\right)^p\\
\bb{p+1}_3(t)&=\q{p+1}_3(t)+t\cdot p\left(\q{1}_1(t)\right)^{p-1}\q{2}_1(t)
\end{align*}
Thus evolution equations become
\begin{align*}
    \dotq{1}(t)&=\du_1(t) \cdot(\pa_{x_1}+t\cdot \pa_{x_2})+\du_2(t)\cdot\pa_{x_2}\\
    \dotq{2}(t)&=-\du_2(t)\cdot\q{1}_1(t)\cdot\pa_{x_2}\\
    \hdots\\
    \dotq{m}(t)&=-\du_2(t)
    \cdot \q{m-1}_1(t)\cdot \pa_{x_2}\\
    \hdots\\
    \dotq{p}(t)&=-\du_2(t)\cdot \q{p-1}_1(t)\cdot\pa_{x_2}-\du_1(t)\cdot p\cdot t \left(\q{1}_1(t)\right)^{p-1}\cdot\pa_{x_3}\\
    \dotq{p+1}(t)&=-\du_2(t)\cdot \q{p}_1(t)\cdot\pa_{x_2}-\du_1(t)\cdot p(p-1)\cdot t \left(\q{1}_1(t)\right)^{p-2}\q{2}_1\cdot\pa_{x_3}+\du_2(t)\cdot p\left(\q{1}_1\right)^p\pa_{x_3}\ .
\end{align*}
Therefore the fields $Y^{(p+1)}_1$ and $Y^{(p+1)}_2$ described in Theorem~\ref{thm:adapted_q} are
\begin{align*}Y^{(p)}_1&=\pa_{\q{1}_1}+t\cdot \pa_{\q{1}_2}-p\cdot t\left(\q{1}_1\right)^{p-1}\cdot\pa_{\q{p}_3}-p(p-1)\cdot t\left(\q{1}_1\right)^{p-2}\q{2}_1\cdot\pa_{\q{p+1}_3} \quad\text{and}\\
Y^{(p)}_2&=\pa_{\q{1}_2}-\sum_{m=1}^{p+1} \q{m-1}_1\cdot \pa_{\q{m}_2}+p \left(\q{1}_1\right)^p \cdot\pa_{\q{p+1}_3}\ .
\end{align*}
\end{example}
\section{Examples and applications}\label{sec:applications}

\subsection{Adapted coordinates for invariant systems on Lie groups}\label{ssec:groups}

\paragraph{The geometric setting}
Let us consider a special situation of a control-linear system \eqref{eqn:lin_ctrl_syst} with $M=G$ being a Lie group, and $\mathcal{D}$ being a left-invariant distribution on $G$, determined by an $l$-dimensional subspace $\spann\{e_1,\hdots,e_l\}$ of the Lie algebra $\g=\T_eG$. For simplicity of notation, in the calculations we will assume that $G$ is a matrix group, and we will denote the matrix multiplication by $\circ$. We shall, however, address a general situation at the end of this subsection. 
\smallskip

In the above setting vector fields spanning $\mathcal{D}$ are simply $X_i(g)=g\circ e_i$, and we may naturally identify a control $u(t)=(u_1(t),u_2(t),\hdots,u_l(t))$ at $t$ with an element $\sum_iu_i(t)e_i\in\g$. Now equation \eqref{eqn:lin_ctrl_syst} reads simply as
\begin{equation}
\label{eqn:cs_group}
    \dot g(t)=g(t)\circ u(t); \qquad g(0)=e,
\end{equation}
where for simplicity we set the initial point to be the group identity $e\in G$. Now a simple calculation shows that, for a given control $\du\in\T_u\U\simeq \U$, curves $\bb{m}(t)$ satisfy the following system of ODEs: 
\begin{equation}\label{eqn:bb_group}
\begin{split}
    \dotbb{1}(t)&=\bb{1}(t)\circ u(t)+g(t)\circ \du(t)\\
\dotbb{2}(t)&=\bb{2}(t)\circ u(t)+2\cdot\bb{1}(t)\circ \du(t)\\
    &\hdots\\
    \dotbb{k}(t)&=\bb{k}(t)\circ u(t)+k\cdot \bb{k-1}(t)\circ \du(t)
\end{split}
\end{equation}

\paragraph{Calculation of adapted coordinates}
Now we would like to apply Theorem~\ref{thm:adapted} and construct adapted coordinates in the above setting. The comparison of the above system \eqref{eqn:bb_group} with general formulas \eqref{eqn:bb_k} reveals that $\D_{g(t)}X_i[b]=b\circ e_i$, and $\D^m_{g(t)}X_i\equiv 0$ for $m>1$. Hence, by the results of Theorem~\ref{thm:adapted}, also $\lin{m}_a(t)\equiv 0$ for all $m>1$. On the other hand, covectors $\lin{1}_a(t)$ should satisfy
$$0=\dotlin{1}_a(t)[b]+\lin{1}_a\left[\sum_i u_i(t)\D_{g(t)}X_i[b] \right]=\dotlin{1}_a(t)[b]+\lin{1}_a\left[b\circ u(t)\right]\ ,$$
with $\lin{1}_a(0)=:\xi_a$ being a fixed basis of $\T^\ast _eG=\g^\ast$. Let us identify $\g$ with $\R^n$ by means of 
$$I:\g\ni b\longmapsto \left(\<\xi_1,b>,\hdots,\<\xi_n,b>\right)\in \R^n\ .$$
Now a simple calculation shows that $\lin{1}(t)[b]:=I^{-1}\left(\lin{1}_1(t)[b],\hdots,\lin{1}_n(t)[b] \right)$ is just 
$$\lin{1}(t)[b]=b\circ g(t)^{-1}\ .$$
By \eqref{eqn:dot_q_k}, we conclude that for each $m=1,2,3,\hdots$ the adapted coordinates $\q{m}(t)=\lin{1}(t)\left[\bb{m}(t)\right]=\bb{m}(t)\circ g(t)^{-1}\in\g$ satisfy
$$\dotq{m}(t)=m\sum_i\du_i(t) \lin{1}(t)\left[\D_{g(t)}X_i\left[\bb{m-1}(t)\right]\right]=m\ \lin{1}(t)\left[\bb{m-1}(t)\circ\du(t)\right]=m\ \bb{m-1}(t)\circ\du(t)\circ g(t)^{-1}.$$
In this particular situation, we may easily get rid of $\bb{m}$'s to arrive at the following system of ODEs:
\begin{equation}\label{eqn:q_group}
\begin{split}
    \dotq{1}(t)&=g(t)\circ \du(t)\circ g(t)^{-1} =\Ad_{g(t)}\du(t)\\
    \dotq{2}(t)&=2\ \bb{1}(t)\circ \du(t)\circ g(t)^{-1}=2\ \q{1}(t)\circ g(t)\circ \du(t)\circ g(t)^{-1} =2\ \q{1}(t)\circ \Ad_{g(t)}\du(t)\\
    &\hdots\\
    \dotq{k}(t)&=k\ \q{k-1}(t)\circ \Ad_{g(t)}\du(t)\ .
\end{split}
\end{equation}
This leads to the following set of solutions
\begin{equation}\label{eqn:b_group_solutions}
\begin{split}
    \bb{1}(t)&=\left[\int_0^t \Ad_{g(t_1)}\du(t_1) \dd t_1\right]\circ g(t) \\
    \bb{2}(t)&=2\left[\int_0^t \q{1}(t_1)\circ \Ad_{g(t_1)}\du(t_1)\dd t_1 \right]\circ g(t)\\
    &\qquad\qquad\qquad=2\left[\int_0^t\int_0^{t_1}\Ad_{g(t_2)}\du(t_2)\dd t_2\circ \Ad_{g(t_1)}\du(t_1)\dd t_1 \right]\circ g(t)\\
    &\hdots\\
    \bb{k}(t)&=k!\left[ \int_0^t\int_0^{t_1}\hdots \int _0^{t_{k-1}}\Ad_{g(t_k)}\du(t_k)\circ\hdots\circ \Ad_{g(t_1)}\du(t_1)\dd t_k\hdots\dd t_1\right]\circ g(t) ,
\end{split}
\end{equation}
which are in perfect agreement with the results of \cite{ELD_new}, where the same formulas were obtained by a different method. 

\paragraph{Remark about general Lie groups} Finally, we would like to comment about the geometric sense of the formulas \eqref{eqn:b_group_solutions} on a general Lie group, where, at the level of the Lie algebra $\g$, there is no obvious analogue of the matrix multiplication $\circ$. To do this observe that our basic equation \eqref{eqn:cs_group} reads as 
\begin{equation}
\label{eqn:cs_group_geometry}
    \dot g(t)=F(g(t),u(t))\ ,
\end{equation}
where $F:G\times \g\ra \T G$ is the left trivialization of $\T G$, i.e. $F(g,a)=\T L_g[a]$, with $L_g:G\ra G$ given by $L_g(h)=gh$. For our purposes, it will be convenient to treat $F$ as a restriction of the map $\T m:\T G\times \T G\ra \T G$ -- the tangent map to the group multiplication $m:G\times G\ra G$ -- to $G\times \T_eG\subset \T G\times \T G$, where on the first leg we embed $G$ into $\T G$ as the zero section. Now system \eqref{eqn:bb_group} describes the $k$-jet of \eqref{eqn:cs_group_geometry}, i.e. it is constructed by feeding $\T^k F:\T^k G\times \T^k\g\lra \T^k\T G$ with the $k$-th jet $(g,\bb{1},\hdots,\bb{k})\in \T^k_gG$, and $(u,\du,0,\hdots,0)\in \T^k\g\equiv \underbrace{\g\times\g\hdots\times\g}_{k+1}$, and composing the result with the canonical flip $\kappa:\T^k\T G\simeq \T\T^k G$ to get a vector tangent to $\T^k G$. Remembering that $F$ is a restriction of $\T m$, we conclude that, up to the canonical identification $\kappa:\T^k\T G\simeq \T \T^kG$, the multiplication in the system \eqref{eqn:bb_group} is the multiplication $\T^k\T m:\T^k\T G\times \T^k\T G\ra\T^k\T G$ applied to a particular pair of $k$-jets embedded in $\T^k\T G$. At this point, it is worth to remark, that $\T^k\T G$ with the multiplication $\T^k\T m$ is actually a Lie group -- see \cite[37.16]{KMS_1993}. \smallskip

Finally, in the setting of Lie groups we pass to adapted coordinates by the right translation by $g(t)^{-1}$ acting on the $k$-jet $(q,\bb{1},\hdots,\bb{k})$, i.e 

$$(e,\q{1},\hdots,\q{k})=\T^k R_{g(t)^{-1}}(g,\bb{1},\hdots,\bb{k})$$

Now since in the multiplication we have $m(R_gf,h)=m(fg,h)=m(f,gh)=m(f,L_gh)$, setting $ (g,\bb{1},\hdots,\bb{k})=\T^kR_{g(t)}(e,\q{1},\hdots,\q{k})$ in \eqref{eqn:bb_group} does not changes the multiplication $\T^k F\subset \T^k \T m$, but translates to applying $\Ad_{g(t)}$ to the element $\du(t)$ and vanishing the $u(t)$-part because of the evolution equations.\smallskip

Summing up, the multiplication on the right-hand side of the system \eqref{eqn:q_group} is the multiplication $\T^k\T m$ of an element  $(e,\q{1},\hdots,\q{k})\in \T^k_eG\subset \T^k\T G$ by an element $(0,\Ad_{g(t)}\du(t),0,\hdots,0)\in\g^{k+1}\equiv\T^k\g\subset \T^k\T G$. This, however, does not mean that we have a multiplication between elements of $\g$. In general, there is no canonical identification of $\T^k_e G$ with $\g^k$ (such an identification requires a choice of local coordinates -- see Remark~\ref{rem:jets} -- or, as happens with matrix groups, an embedding of $G$ into $M_N(\R)$). Thus the resulting formulas \eqref{eqn:b_group_solutions} have only a local sense. 

\subsection{Optimality conditions in sub-Riemannian geometry}\label{ssec:sr}

In Sec.~\ref{sec:coordinates} we were able to associate a  control-affine system \eqref{eqn:aff_ctr_sys} with the problem of calculating the derivatives (or, more precisely, the jet-expansion) of the end-point map at a given trajectory of a control-linear system \eqref{eqn:lin_ctrl_syst}. The results of  Sussmann and Jurdjevic \cite{Sussmann_Jurdjevic_1972} give simple criteria for local non-controllability of a control system at a given point. It turns out that for $k=1$ and $k=2$ these criteria are closely related to necessary optimality conditions in sub-Riemannian geometry. Below we shall explain this relation.

\paragraph{Sub-Riemannian geodesic problem}
Consider the control-linear system \eqref{eqn:lin_ctrl_syst} with $\Omega=L^2([0,T],\R^l)$. For a trajectory $q(t)$, which corresponds to a control $u\in L^2([0,T],\R^l)$, we define its \emph{energy} at $t$ as
$$E^t(u):=\frac 12\int_0^t \sum_i u_i(\tau)^2\dd \tau\ . $$
Given $T>0$, and a pair of points $q_0,q_1\in M$, the question is to find a trajectory of \eqref{eqn:lin_ctrl_syst} which joins $q_0=q(0)$ with $q_1=q(T)$ while minimizing the energy $E^T$. Solutions of such a problem are called \emph{minimizing sub-Riemannian geodesics} -- see \cite{Montgomery_2006, Rifford_2014} for more details.\medskip

There is a fundamental relation between optimality in the above sense, and the properties of the \emph{extended end-point map} $\END^t:\U\lra M\times\R$; defined by $\END^t(u):=(\End^t(u),E^t(u))$.

\begin{fact} If $\END^T$ is open at $u$, then the corresponding trajectory  cannot be a sub-Riemannian geodesics. 
\end{fact}

 All the known optimality criteria  for sub-Riemannian geodesics, with the sole exception of \cite{Hakavuori_LeDonne_2016}, depend on the above result. They use some version of a (higher-order) open mapping theorem (see \cite{Sussmann_OMT} for a set of general results of this type) linking the properties of (first or higher) derivatives of the end-point map $\END^T$ with the openness of $\END^T$ itself. In what follows we shall revise some of these classical results using the construction of adapted coordinates from Theorems~\ref{thm:adapted} and \ref{thm:adapted_q}.

\paragraph{Sub-Riemannian optimality conditions of degree one}
It follows easily from the standard open mapping theorem that if $\image \D_u\END^T$ (which is of course a vector space) equals $\T_{(q(T),E^T(u))}(M\times\R)$, then $\END^T$ is open at $u$, and hence the related trajectory $q(t)$ cannot be optimal. Thus an obvious necessary condition for optimality is that $\image \D_u\END^T$ should be a proper vector subspace in the tangent space $\T_{(q(T),E^T(u))}(M\times\R)$. Curves satisfying this condition are called  \emph{extremals}. Among all extremals we distinguish a subclass of \emph{abnormal trajectories}. 
\begin{df} A trajectory $q(t)$ of \eqref{eqn:lin_ctrl_syst} corresponding to a control $u\in L^2([0,T],\R^l)$ is called \emph{abnormal} if $\image \D_u\End^T$ is a proper vector subspace of the tangent space $\T_{q(T)}M$. In particular, an abnormal trajectory is an extremal. 
\end{df}

Below we will concentrate solely on this class of curves as they are the most interesting (and mysterious) type of extremals \cite{Montgomery_2006}. The basic reason for this is that an abnormal trajectory is an extremal regardless of the properties of the energy functional, thus being abnormal is a geometric property of the control system itself. 

\begin{lem}
\label{lem:results_deg_1}
Let $q(t)$, with $t\in[0,T]$, be an abnormal SR trajectory  corresponding to the control $u\in L^2([0,T],\R^l)$. Then the related  control-affine system $(\Sigma^{(1)})$ in $\R\times\R^n\ni (t,\q{1})$ described in Thm~\ref{thm:adapted_q} is not controllable at any point $(t,0)$. In particular, there exist an $n$-tuple $(\phi^1,\hdots,\phi^n)\in \R^n\setminus\{\bm 0\}$ such that the covector $\phi(t)=\sum_a \phi^a\lin{1}_a(t)\in \T^\ast_{q(t)}M$ satisfies the following condition
    $$\<\phi(t),X_i(q(t))>=0\quad \text{for all $i=1,2,\hdots,l$ and all $t\in[0,T]$}.$$
\end{lem}
Observe that due to \eqref{eqn:ew_Phi_1}, the covector $\phi(t)$ described above is a \emph{Pontryagin covector}, i.e  its evolution is compatible with the evolution of \eqref{eqn:lin_ctrl_syst} corresponding to $u$:
$$\<\dot\phi(t), b>+\<\phi(t),\sum_i u_i(t)\cdot \D_{q(t)} X_i[b]>=0\quad\text{ for every $b\in \T_{q(t)}M$.}$$
\begin{proof}
By definition, since $q(t)$ is an abnormal trajectory, we know that $\image \D_u\End^T=\{\bb{1}(T,\du)\that \du\in L^2([0,T],\R^l)\}$ is a proper vector subspace of $\T_{q(T)}M$. Passing to the adapted coordinates $\q{1}\in\R^n$ -- which in degree one involves a linear transformation -- we conclude that the space $V:=\{\q{1}(T,\du)\that \du\in L^2([0,T],\R^l)\}$ is a proper vector subspace of $\R^n$. However, by the results of Thm~\ref{thm:adapted_q}, space $V$ is just  $\mathcal{R}_T$ -- the reachable  set at time $T$ of the control system 
$$\dotq{1}(t)=\sum_i\du_i(t)\cdot Y^{(1)}_i(t)=\sum_i\du_i(t)\cdot \lin{1}(t)[X_i(q(t))];\qquad \q{1}(0)=0.
$$
 Clearly, by extending a control $\du\in L^2([0,\tau],\R^l)$ by zero on $[\tau,T]$, also $\mathcal{R}_\tau\subset\mathcal{R}_T=V$ for every $\tau\in [0,T]$. We conclude that the fields $Y^{(1)}_i(t)$ must be tangent to $V$ for all $t\in[0,T]$. Hence, each non-zero covector  $\phi=(\phi^1,\hdots,\phi^n)\in(\R^n)^\ast\simeq 
 \R^n$ annihilating $V$, will satisfy 
$0=\<\phi,Y^{(1)}_i(t)>=\<\phi,\lin{1}(t)\left[X_i(q(t))\right]>$ for every $i=1,2,\hdots,l$ and every $t\in[0,T]$. This ends the proof.
\end{proof}

\paragraph{Optimality conditions of degree two}
A basis of our considerations in this part is the following result

\begin{thm}[Agrachew-Sarychew Index Lemma \cite{Agrachev_Sarychev_1996}]\label{thm:as_lemma}
    Let $q(t)$, with $t\in[0,T]$ be an abnormal minimizing sub-Riemannian geodesics corresponding to the control $u\in L^2([0,T],\R^l)$, and denote by $r$ the codimension of $\image \D_u\End^T$ in $\T_{q(T)}M$. Then there exists a covector $\phi_0\in \T^\ast_{q(T)}M$ with the following properties 
    \begin{itemize}
        \item $\phi_0$ annihilates the image $\image \D_u\END^T$
        \item the negative index of the quadratic map 
        $$\<\phi_0,\D^2_u\END^T>\big|_{\ker \D_u\END^T}:\ker \D_u\END^T\lra \R$$
        is at most $r-1$. 
    \end{itemize}
 Above we treat a covector $\phi_0\in \T^\ast M$, as an element  $\phi_0+0\cdot\dd r\in  \T^\ast(M\times \R)$ -- hence its action on $\T \R$-vectors is trivial. 
\end{thm}

The above result is stated in terms of the extended end-point map $\END^T$. It is however easy to reformulate it in the language of the standard end-point map $\End^T$ for a moderate price of rising the index by one. 

\begin{prop}[modification of the Agrachew-Sarychew Index Lemma]\label{lem:as_plus}
    Let $q(t)$, with $t\in[0,T]$, be an abnormal minimizing sub-Riemannian geodesics corresponding to the control $u\in L^2([0,T],\R^l)$, and denote by $r$ the codimension of $\image \D_u\End^T$ in $\T_{q(T)}M$. Then there exists a covector $\phi_0\in \T^\ast_{q(T)}M$ with the following properties
    \begin{itemize}
        \item $\phi_0$ annihilates the image $\image \D_u\End^T$
        \item the negative index of the quadratic map 
        $$\<\phi_0,\D^2_u\End^T>\big|_{\ker \D_u\End^T}:\ker \D_u\End^T\lra \R$$
        is at most $r$. 
    \end{itemize}
\end{prop}
\begin{proof} 
    Let $\phi_0$ be as in the assertion of Theorem~\ref{thm:as_lemma}. Since $\phi_0$ annihilates the $\T\R$-direction, then clearly 
    $$\<\phi_0,\D^2_u\END^T>=\<\phi_0,\D^2_u\End^T>\ .$$
    Now as $\ker \D_u\END^T=\ker\D_u\End^T\cap \ker\D_u E^T$ and $E^T$ is an $\R$-valued map, the space $\ker \D_u\END^T$ is a subspace of $\ker \D_u\End^T$ of codimension less or equal 1. Thus if $\<\phi_0,\D^2_u\End^T>$ would be negatively defined on some $r+1$ dimensional vector subspace $W\subset \ker \D_u\End^T$, then $W'=W\cap \ker \D_u E^t\subset \ker \D_u\END^T$ would be a vector subspace of dimension at least $r$ on which  $\<\phi_0,\D^2_u\End^T>=\<\phi_0,\D^2_u\END^T>$ is negatively defined. This contradicts the assertion of Theorem~\ref{thm:as_lemma}.
\end{proof}

Under a further assumption that the index mentioned in Prop.~\ref{lem:as_plus} is actually zero, results of Theorem~\ref{thm:adapted_q} allow to derive the following \emph{Goh conditions}.

\begin{lem}\label{lem:results_deg_2}
Let $q(t)$ be a sub-Riemannian trajectory  corresponding to a control $u\in L^2([0,T],\R^l)$. Assume that a covector $\phi_0\in \T^\ast_{q(T)}M$ satisfies the conditions 
\begin{enumerate}[(i)]
    \item $\phi_0$ annihilates the image $\image \D_u\End^T$
    \item the negative index of the quadratic map $$\<\phi_0,\D^2_u\End^T>\big|_{\ker \D_u\End^T}:\ker \D_u\End^T\lra \R$$
        is zero. 
\end{enumerate}
Then there exists an $n$-tuple $(\phi^1,\hdots,\phi^n)\in \R^n\setminus\{\bm 0\}$ such that the Pontryagin covector $\phi(t)=\sum_a \phi^a\lin{1}_a(t)\in \T^\ast_{q(t)}M$ satisfies the conditions 
$$\<\phi(t),X_i(q(t))>=0\quad\text{and}\quad\<\phi(t),[X_i,X_j](q(t))>=0\qquad \text{for all $i,j=1,2,\hdots,l$ and all $t\in[0,T]$.} $$
\end{lem}
\begin{proof}
By assumption, for every controls $\ddu{1},\ddu{2}\in L^2([0,T],\R^l)$ such that $\D_u\End^T\left[\ddu{1}\right]=\bb{1}(T,\ddu{1})=0$ we have
$$\<\phi_0,  \frac 12\D^2_u\End^T\left[\ddu{1}\right]=\bb{2}(T,\ddu{1})>\geq 0\quad\text{and}\qquad\<\phi_0,\D_u\End^T\left[\ddu{2}\right]=\bb{1}(T,\ddu{2})>= 0\ .$$
Now observe that, since $\bb{1}(T,\ddu{1})=0$, when passing to adapted coordinates formulas \eqref{eqn:q_1} and \eqref{eqn:q_2} give us
$$\q{2}(T,\du_1)=\lin{1}(T)\left[\bb{2}(T,\ddu{1})\right]\quad\text{and}\quad \q{1}(T,\ddu{1})=\lin{1}(T)\left[\bb{1}(T,\ddu{2})\right]\ .$$
Hence for a non-zero covector $\phi_0'=\lin{1}(T)^\ast \phi_0\in \R^n$, we have
$\<\phi_0', \q{2}(T,\ddu{1})>\geq 0$ and $\<\phi_0',\q{1}(T,\ddu{2})>= 0$, and thus
$$\<\phi_0', \q{2}(T,\ddu{1})+\q{1}(T,\ddu{2})>\geq 0\qquad\text{whenever $\q{1}(T,\ddu{1})=\lin{1}(T)\left[\bb{1}(T,\ddu{1})\right]=0$}.  $$
In light of Theorem~\ref{thm:adapted_q} this translates as a condition that the control system
\begin{equation}\label{eqn:cs_deg2}
\begin{split}
    \dot{x}^{(1)}&=\sum_i\ddu{1}_i(t)\cdot Y^{(1)}_i(t) \\
    \dot{x}^{(2)}&=\sum_i\ddu{1}_i(t)\cdot   Y^{(2)}_i(t,x^{(1)}) +\sum_i\ddu{2}_i(t) \cdot Y^{(1)}_i(t)\ ;
\end{split}
\end{equation}
i.e. ${x}^{(1)}(t)=\q{1}(t,\ddu{1})$ and ${x}^{(2)}(t)=\q{2}(t,\ddu{1})+\q{1}(t,\ddu{2})$; is not controllable (as a system in $\R^n\times\R^n$, but also in  $V\times \R^n$, where $V:=\{\q{1}(T,\du)\that \du\in L^2([0,T],\R^l)\}$ is the reachable set of coordinates $x^{(1)}$ -- as in the proof of Lemma~\ref{lem:results_deg_1}) at time $T$ at the point $(x^{(1)}=0,x^{(2)}=0)$. In particular, by the criteria of local controllability of Jurdjevic-Sussmann \cite{Sussmann_Jurdjevic_1972}, the control fields 
$$\Z{2}_i(t,x^{(1)}):=(Y^{(1)}_i(t),Y^{(2)}_i(t,x^{(1)}))\quad\text{and}\quad \Z{1}_i:=(0,Y^{(1)}_i(t))$$ and their Lie brackets calculated at $(T, x^{(1)}=0,x^{(2)}=0 )$ are contained in a  proper linear subspace $W\subset V\times\R^n$. Note also, that the reachable set of the discussed control system at any time $T'\in[0,T]$ at the point $(x^{(1)}=0,x^{(2)}=0)$ is naturally contained in  the analogous reachable set at $T$ (by simply extending the controls by zero on $[T',T]$). Thus we may assume that $W$ is spanned by  the control vector fields and their brackets calculated at  $(T', x^{(1)}=0,x^{(2)}=0 )$ for all $T'\in[0,T]$. Further, as the system $\dot x^{(1)}=\sum_i \ddu{1}_i\cdot Y^{(1)}_i(t)$ is controllable in $V\subset\R^n$, we may assume that $W=V\times V'$, where $V'\subset \R^n$ is a proper vector subspace. \smallskip

Finally, a simple calculation shows that at time $t$
$$[\Z{2}_i,\Z{2}_j]= (0,\lin{1}(t)\Big[[X_i,X_j](q(t))\Big])\ ,$$
hence, in particular, for every $t\in[0,T]$ fields $(0,\lin{1}(t)\big[[X_i,X_j](q(t))\big])$ and $(0,Y^{(1)}_i(t))=(0,\lin{1}(t)\big[X_i(q(t))\big])$ belong to $V\times V'$. Taking $(\phi^1,\hdots,\phi^n)\in\R^n\setminus\{\bm 0\}$ to be any covector annihilating $V'$ we get the assertion.
\end{proof}

 Finally, we may refer to the recent results of \cite{MJ_SR_deg2} where it is proved that it is possible to put the Agrachew-Sarychew index to zero by dividing the curve into a finite number of pieces.

\begin{lem}[\cite{MJ_SR_deg2}]\label{lem:parts}
 Let $q(t)$, with $t\in[0,T]$ be an abnormal minimizing sub-Riemannian geodesics corresponding to the control $u(t)$, and denote by $r$ the codimension of $\image \D_u\End^T$ in $\T_{q(T)}M$. Then there exists at most $r$ points $0\leq\tau_1<\tau_2<\hdots< \tau_s\leq T$ such that on every subinterval $[a,b]\subset (\tau_i,\tau_{i+1})$ for $i=1,2,\hdots, s-1$ the trajectory $q(t)$ with $t\in[a,b]$ satisfies the assumptions of Lemma~\ref{lem:results_deg_2}.\smallskip 

In particular, on each of the pieces the assertion of Lemma~\ref{lem:results_deg_2}  holds. Thus we get \emph{Goh conditions} on each piece (perhaps for different Pontryagin covectors). 
\end{lem}

Actually, the tool of adapted coordinates can be used to give a more refined picture of the second-order optimality conditions in sub-Riemannian geometry. We refer to \cite{MJ_SR_deg2} for more details.

\section{General polynomial variations of the end-point map}

\paragraph{Motivations}
Curves $\bb{k}(t,\du)$ (and their counterparts $\q{k}(t,\du)$ in adapted coordinates) are obtained by studying reactions of the end-point map $\End^t$ to changes of the control argument of the form $s\mapsto u+s\cdot\du$. It is, however, most natural to study such reactions for a more general family of controls, i.e.
\begin{equation}
\label{eqn:variation_control}
s\longmapsto u+s\cdot \ddu{1}+s^2\cdot \ddu{2}+\hdots+s^k\cdot \ddu{k}\ .
\end{equation}
Geometrically this corresponds to studying the natural lift of the end-point map to the space of $k$-jets, i.e. $\T^k\End^t:\T^k_u\Omega \lra \T^k_{q(t)} M$.

Let us observe that such a generalization looks promising from the point of view of applications. First of all, recall formula \eqref{eqn:expansion_End_2_degree}, which allowed us to understand Agrachev's approach to the study of second derivatives. Further note that in our interpretation of sub-Riemannian optimality conditions of degree two in Lemma~\ref{lem:results_deg_2}, it turned out that the non-controllability of the system \eqref{eqn:cs_deg2} played a crucial role. This system is clearly not constructed from a variation of the end-point map related with the control shift $s\mapsto u+s\cdot\du$, as it involves two controls $\ddu{1},\ddu{2}$ instead of a single control $\du$. However, it is not difficult to guess that it corresponds to variation by the family of controls
$s\mapsto u+s\cdot \ddu{1}+s^2\cdot \ddu{2}$ undergoing a transformation of adapted coordinates. 

\paragraph{Polynomial variations of the end-point map}
Family of controls \eqref{eqn:variation_control} is encoded by a $k$-tuple $\vv{\du}:=(\ddu{1},\ddu{2},\hdots,\ddu{k})\in \U^k$, establishing a canonical isomorphism between the fibre of the $k$-jet bundle $\T^k_u\U$ and $\U^k$ (true for every vector space). To simplify the notation, let us denote
$$s\circ\vv{\du}:=s\cdot\ddu{1}+s^2\ddu{2}+\hdots+s^k\cdot\ddu{k}\ .$$
By a \emph{$k$-variation}  of $\End^t$ at $u$ in the direction of $\vv{\du}$  we shall understand the $k$-jet of $s\mapsto \End^t[u+s\circ\vv{\du}]$. Analogously to our previous considerations, in a local coordinate system it can be encoded by a family of curves $\bb{i}(t,\vv{\du})$, for $i=1,2,\hdots,k$,
\begin{align*}
&\End^t[u+s\circ\vv{ \du}]\overset{loc}=q(t)+s\cdot \bb{1}(t,\vv{\du})+s^2\cdot \bb{2}(t,\vv{\du})+\hdots+s^k\cdot \bb{k}(t,\vv{\du})+o(s^k)\ ,
\end{align*}
where $\bb{m}(t,\vv{\du})\overset{loc}=\frac 1{m!}\cdot \frac{\pa^m}{\pa s^m}\End^t[u+s\circ\vv{\du}]$. 
\smallskip

Let us remark that using Lemma~\ref{lem:composition} it is possible to express curves $\bb{m}(t,\vv{\du})$ in terms of the derivatives $\D^m_u\End^t$ evaluated on various controls $\ddu{s}$ forming the $k$-tuple $\vv{\du}$. Indeed 
\begin{align*}
    &\End^t[u+s\circ\vv{ \du}]\overset{loc}{=}\\
    &\End^t[u]+ \D_{u}\End^t[s\circ\vv{ \du}]+\frac 1{2!} \D^2_{u}\End^t[s\circ\vv{ \du}]+\hdots+\frac 1{k!} \D^k_{u}\End^t[s\circ\vv{ \du}]+o(s^k)=\\
    &\End^t[u]+s\cdot \D_{u}\End^t[\du^{(1)}]+s^2\cdot\left(\frac 1{2!} \D^2_{u}\End^t[\du^{(1)},\du^{(1)}]+\D_{u}\End^t[\du^{(2)}]\right)+\hdots\\
    &\qquad +s^k\cdot\left( \frac 1{k!} \D^k_{u}\End^t[\du^{(1)},\hdots,\du^{(1)}]+\frac 1{(k-1)!} \D^{k-1}_{u}\End^t[\du^{(2)},\du^{(1)},\hdots \du^{(1)}]+\hdots \right.\\
    &\left.\qquad \qquad +\frac 1{2!}\D^2_u\End^t[\du^{(1)},\du^{(k-1)}]+ \D_{u}\End^t[\du^{(k)}] \right)+o(s^k)=
\end{align*}
Leading to
\begin{align*}
    \bb{1}(t,\vv{\du})=&\D_u\End^t[\ddu{1}]\\
    \bb{2}(t,\vv{\du})=&2\D_u\End^t[\ddu{2}]+\D^2_u\End^t[\ddu{1},\ddu{2}]\\
    \hdots\\
    \bb{k}(t,\vv{\du})=&k!\cdot\D_u\End^t[\ddu{k}]+\frac{k!}{2!}\D^2_u\End^t[\ddu{1},\ddu{k-1}]+\hdots\D^k_u\End^t[\ddu{1},\hdots,\ddu{1}]\ .
\end{align*}

\paragraph{Description of curves $\bb{m}(t,\vec{\du})$}
It is possible to generalize our previous results about $k$-variations in the direction of $\du$  to describe curves $\bb{m}(t,\vv{\du})$.

\begin{lem}[the general form of $\bb{m}(t,\vv{\du})$s]
\label{lem:b_k_polynomial}
 Consider a $k$-tuple  $\vv{du}=(\ddu{1},\hdots,\ddu{k})\in \U^k\simeq\T^k_u\U$. Let the $k$-jet of the curve $s\mapsto \End^t[u+s\circ\vv{\du}]$ at $s=0$ be given in local coordinates by 
$$\End^t(u+s\circ\vv{\du})\overset{loc}=q(t)+s\cdot \bb{1}(t,\vv{\du})+s^2\cdot\bb{2}(t,\vv{\du})+\hdots s^k\cdot\bb{k}(t,\vv{\du})
+o(s^k)\ .$$
Then curves $\bb{m}(t,\vv{\du})$
 are subject to the following  system of ODEs:
 \begin{equation}
\bb{m}(t,\vv{\du})=\sum_{r=0}^m\sum_i\du_i^{(r)}(t)\left(\sum_{\alpha,\  w(\alpha)=m-r} \frac 1{\alpha!} \cdot\D^{|\alpha|}_{q(t)}X_i[\vec{b}(t,\du)^\alpha]  \right) \ ,    
 \end{equation} 
 where we use the notation introduced on page \pageref{multi_ind_not} and we denote $\ddu{0}=u$.  
\end{lem}

\begin{proof}
It is enough to repeat the proof of Lemma~\ref{lem:b_k}. Let $q_s(t)$ be a solution of 
$$\dot{q_s}(t)=\sum_i\left(u_i(t)+s\cdot\ddu{1}_i(t)+\hdots +s^k\cdot \ddu{k}_i(t)\right)X_i(q_s(t))\ .$$
Then analogously as before 
$$\pa_s^{m}\Big|_0\left[X_i(q_s(t))\right]=m!\cdot\left(\sum_{\alpha,\  w(\alpha)=m} \frac 1{\alpha!} \cdot\D^{|\alpha|}_{q(t)}X_i[\vec{b}(t,\vv{\du})^\alpha]  \right)\ .$$
Further, note that for $r>m$ we have
$\pa_s^{m}\Big|_0\left[s^r\cdot X_i(q_s(t))\right]=0$, while for $r\leq m$
$$\pa_s^{m}\Big|_0\left[s^r\cdot X_i(q_s(t))\right]=
r!\cdot \binom{m}{r}\cdot \pa^{m-r}\left[X_i(q_s(t))\right]=m!\cdot\left(\sum_{\alpha,\  w(\alpha)=m-r} \frac 1{\alpha!} \cdot\D^{|\alpha|}_{q(t)}X_i[\vec{b}(t,\vv{\du})^\alpha]  \right)$$
Now for $m\leq k$ we have
\begin{align*}
    \dotbb{m}&(t,\vv{\du})= \frac 1{m!}\cdot\frac{\dd}{\dd t}\left(\pa_s^m\Big|_0q_s(t)\right)=\frac 1{m!}\cdot\pa_s^m\Big|_0\left(\dot{q_s}(t)\right)=\\
    &\frac 1{m!}\cdot\pa_s^m\Big|_0\left[\sum_i \left(u_i+s\ddu{1}_i+s^2\ddu{2}_i\hdots +s^k\ddu{k}_i\right)X_i(q_s(t))\right]=\\
    & \frac 1{m!}\cdot\sum_i \bigg\{ u_i(t)\cdot\pa_s^m\Big|_0\left[X_i(q_s(t))\right]+\ddu{1}_i(t)\cdot\pa_s^m\Big|_0\left[s\cdot X_i(q_s(t))\right]+\hdots \ddu{k}_i(t)\cdot\pa_s^m\Big|_0\left[s^k\cdot X_i(q_s(t))\right]\bigg\}=\\ &\sum_iu_i(t)\cdot\left(\sum_{\alpha,\  w(\alpha)=m} \frac 1{\alpha!} \cdot\D^{|\alpha|}_{q(t)}X_i[\vec{b}(t,\vv{\du})^\alpha]  \right)+\\
    &\phantom{XX}\sum_i  \ddu{1}_i(t)\cdot\left(\sum_{\alpha,\  w(\alpha)=m-1} \frac 1{\alpha!} \cdot\D^{|\alpha|}_{q(t)}X_i[\vec{b}(t,\vv{\du})^\alpha]  \right)+\\
    &\phantom{XX}\sum_i \ddu{2}_i(t)\cdot\left(\sum_{\alpha,\  w(\alpha)=m-2} \frac 1{\alpha!} \cdot\D^{|\alpha|}_{q(t)}X_i[\vec{b}(t,\vv{\du})^\alpha]  \right)+\hdots +\\
    &\phantom{XX}+\sum_i \ddu{m}_i(t)\cdot\left(\sum_{\alpha,\  w(\alpha)=0} \frac 1{\alpha!} \cdot\D^{|\alpha|}_{q(t)}X_i[\vec{b}(t,\vv{\du})^\alpha]  \right)=\\
    &=\sum_{r=0}^m\ddu{r}_i(t)\left(\sum_{\alpha,\  w(\alpha)=m-r} \frac  1{\alpha!} \cdot\D^{|\alpha|}_{q(t)}X_i[\vec{b}(t,\vv{\du})^\alpha]  \right).
\end{align*}
This ends the proof.
\end{proof}

\paragraph{Adapted coordinates for polynomial $k$-variations} It is interesting to see how the $k$-variation of $\End^t$ in the direction of $\vv{\du}$ looks in adapted coordinates.

\begin{thm} 
\label{thm:adapted_polynomial}
Consider a $k$-tuple  $\vv{\du}=(\ddu{1},\hdots,\ddu{k})\in \U^k\simeq \T^k_u\U$. Let the $k$-jet of the curve $s\mapsto \End^t[u+s\circ\vv{\du}]$ at $s=0$ be given in local coordinates by 
$$\End^t(u+s\circ\vv{\du})\overset{loc}=q(t)+s\cdot \bb{1}(t,\vv{\du})+s^2\cdot\bb{2}(t,\vv{\du})+\hdots s^k\cdot\bb{k}(t,\vv{\du})
+o(s^k)\ .$$
Let $(\q{1}(t,\vv{\du}),\hdots,\q{k}(t,\vv{\du}))$ be the image of the $k$-jet $(\bb{1}(t,\vv{\du}),\hdots,\bb{k}(t,\vv{\du}))$ under the transformation of adapted coordinates $\AddCoord{k}(t)$. That is,
for any $m=1,2,\hdots k$ we have 
\begin{equation}
    \label{eqn:q_k_polynomial}
    \q{m}(t,\vv{\du}):=\sum_{\alpha,\ w(\alpha)=m} \frac 1{\alpha!} \cdot \lin{|\alpha|}(t) [\vec{b}(t,\vv{\du})^{\alpha}]\ .
    \end{equation}  
    Then, curves $\q{m}(t,\vv{\du})$ satisfy the following  system of ODEs:
\begin{equation}
\label{eqn:dot_q_k_polynomial}
    \dotq{m}(t,\vv{\du})=\sum_{r=1}^m\sum_i\du_i^{(r)}(t)\cdot\left\{\sum_{\alpha,\ \beta,\ w(\alpha)+w(\beta)=m-r} \frac{ 1}{\alpha!} \cdot  \frac 1{\beta !} \cdot  \lin{|\alpha|+1}(t) \left[\D^{|\beta|}_{q(t)}X_i[\vec{b}(t,\vv{\du})^\beta]
,\vec{b}(t,\vv{\du})^{\alpha}\right]\right\}\ .
\end{equation}
\end{thm}

By expressing $\bb{m}(t,\vv{\du})$s in \eqref{eqn:dot_q_k_polynomial} via $\q{m}(t,\vv{\du})$s (i.e. using the inverse of the transformation of adapted coordinates) we get an immediate corollary, generalizing Theorem~\ref{thm:adapted_q}

\begin{thm}
\label{thm:adapted_q_polynomial}
Consider a trajectory $q(t)$ of a control-linear system \eqref{eqn:lin_ctrl_syst} corresponding to the control $u\in \U$ and choose $k\in \mathbb{N}$. Then for every $k$-tuple  $\vv{\du}=(\ddu{1},\hdots,\ddu{k})\in \U^k\simeq \T^k_u\U$ the curve 
$$p^{(k)}(t):=(t, \q{1}(t,\vv{\du}),\hdots,\q{k}(t,\vv{\du}))\ ,$$
where $(\q{1}(t,\vv{\du}),\hdots,\q{k}(t,\vv{\du}))$ is the image under the transformation of adapted coordinates $\AddCoord{k}(t)$  of the $k$-jet of $s\mapsto \End^t[u+s\circ\vv{\du}]$ at $s=0$, is a trajectory  of the control-affine system 
\begin{equation}
\label{eqn:aff_ctr_sys_polynomial}
    \dot p^{(k)}(t)=\pa_t+\sum_{r=1}^k\sum_i \ddu{r}_i(t)\cdot Y_i^{(k-r+1)}(p^{(k)}(t))\qquad\text{with  $p^{(k)}(0)=(0,0,\hdots,0)$.}
\end{equation}
Here, for $m=1,2,\hdots,k$ fields $Y^{(m)}_i$ are defined precisely as in the assertion of Theorem~\ref{thm:adapted_q}.
\end{thm}

Let us now proceed with the proof of the former theorem. 
\begin{proof}[Proof of Theorem~\ref{thm:adapted_polynomial}]
Again we roughly repeat steps from the proof of Theorem~\ref{thm:adapted}. Since $\q{m}(t,\vv{\du})$ is given by \eqref{eqn:q_k_polynomial}, we have
$$\dotq{m}(t,\vv{\du}):=\sum_{\alpha,\ w(\alpha)=m} \frac 1{\alpha!} \cdot \dotlin{|\alpha|}(t) [\vec{b}(t,\vv{\du})^{\alpha}]+\sum_{\alpha,\ w(\alpha)=m} \sum_{l\leq m}\frac 1{\alpha!} \cdot a_l\cdot \lin{|\alpha|}(t) [\dotbb{l},\vec{b}(t,\vv{\du})^{\alpha-1_l}]\ ,$$
which has the form
    $$\dotq{m}(t,\vv{\du})=\sum_i u_i\cdot A_{i} +\sum_{r=1}^m\sum_i \ddu{r}_i\cdot B_{i}^{(r)} \ ,$$
    Like in the proof of Theorem \ref{thm:adapted}
\begin{align*} \sum_i u_i&\cdot A_{i}=\\
    &\sum_{\alpha,\ w(\alpha)=j} \dotlin{|\alpha|} [\vec{b}(t)^{\alpha}]+\sum_{i} u_i\cdot\left\{\sum_{\gamma,\ \beta,\ w(\gamma)+w(\beta)=j} \frac{1}{\gamma!} \cdot  \frac 1{\beta !} \cdot  \lin{|\gamma|+1} [\D^{|\beta|}_{q(t)}X_i[\vec{b}(t)^\beta],\vec{b}(t)^{\gamma}]\right\}=0\ .
\end{align*}
On the other hand
\begin{align*}
     \sum_i\du_i^{(r)}\cdot B_{i}^{(r)}= &\sum_{\alpha,\ w(\alpha)=m}\sum_{l\leq m} \frac{1}{\alpha!} \cdot a_l\cdot \lin{|\alpha|} [\text{part of $\dotbb{l}$ linear in $\du_i^{(r)}$'s},\vec{b}(t)^{\alpha-1_l}]=\\
    &\sum_{\alpha,\ w(\alpha)=m}\sum_{l\leq m} \frac{1}{\alpha!} \cdot a_l\cdot \lin{|\alpha|} \left[\sum_{i} \du_i^{(r)}\cdot\left\{\sum_{\beta,\  w(\beta)=l-r} \frac {\new 1}{\beta !} \cdot\D^{|\beta|}_{q(t)}X_i[\vec{b}(t)^\beta]\right\}
    ,b(t)^{\alpha-1_l}\right]=\\
    &\sum_{i} \ddu{r}_i\cdot\left\{\sum_{\alpha,\ w(\alpha)=m}\sum_{l\leq m} \frac{1}{(\alpha-1_l)!} \cdot \sum_{\beta,\  w(\beta)=l-r} \frac {\new 1}{\beta !} \cdot  \lin{|\alpha-1_l|+1} \left[\D^{|\beta|}_{q(t)}X_i[\vec{b}(t)^\beta]
    ,\vec{b}(t)^{\alpha-1_l}\right]\right\}
\end{align*}

Now observe that a triple $(\alpha,l,\beta)$ where $w(\alpha)=j$, $a_l>0$ and $w(\beta)=l-r$ uniquely determines a pair of multi-indexes $(\gamma=\alpha-1_l,\beta)$ satisfying $w(\gamma)+w(\beta)=(m-l)+l-r=m-r$. Thus we may change the summation order in the expression above to obtain    
$$\sum_i \du_i^{(r)}\cdot B_{i}^{(r)}=\sum_{i} \ddu{r}_i\cdot\left\{\sum_{\gamma,\ \beta,\ w(\gamma)+w(\beta)=m-r} \frac{1}{\gamma!} \cdot  \frac {\new 1}{\beta !} \cdot  \lin{|\gamma|+1} [\D^{|\beta|}_{q(t)}X_i[\vec{b}(t)^\beta]
,\vec{b}(t)^{\gamma}]\right\}\ .$$
This ends the proof. \end{proof}

\paragraph{Example}
Consider 3-variation in the direction of $\vv{\du}=(\ddu{1},\ddu{2},\ddu{3})$. In that situation, it is quite easy to express curves $\q{m}(t,\vv{\du})$ for $m=1,2,3$ in terms of curves $\q{m}(t,\ddu{j})$. It turns out that
\begin{align*}
    \q{1}(t,\vv{\du})&=\q{1}(t,\ddu{1})\\
    \q{2}(t,\vv{\du})&=\q{2}(t,\ddu{1})+\q{1}(t,\ddu{2})\\
    \q{3}(t,\vv{\du})&=\q{3}(t,\ddu{1})+2\q{2}(t,\ddu{1},\ddu{2})+\q{1}(t,\ddu{3})\ .
\end{align*}
Here symbol $\q{2}(t,\ddu{1},\ddu{2})$ denotes the symmetrization 
$$2\q{2}(t,\ddu{1},\ddu{2}):=\q{2}(t,\ddu{1}+\ddu{2})-\q{2}(t,\ddu{1})-\q{2}(t,\ddu{2})\ .$$

\paragraph{A hypothesis related with Goh-conditions}

By the results of 
Theorem \ref{thm:adapted_q_polynomial} system \eqref{eqn:aff_ctr_sys_polynomial} is a control system in $\R\times\underbrace{\R^n\times\hdots\R^n}_k\ni (t,\q{1},\q{2},\hdots,\q{k})$ controlled by a family of vector fields
\begin{align*}
\Z{k}_i=&(Y_i^{(1)}(t),Y_i^{(2)}(t,p^{(1)}),\hdots,Y_i^{(k)}(t,p^{(1)},\hdots,p^{(k-1)})),\\
\Z{k-1}_i=&(0,Y_i^{(1)}(t),\hdots,Y_i^{(k-1)}(t,p^{(1)},\hdots,p^{(k-2)})),\\
 &\hdots \\
 \Z{2}_i=&(0,\hdots,0,Y_i^{(1)}(t),Y^{(2)}_i(t,p^{(1)})),\\
 \Z{1}_i=&(0,\hdots,0,Y_i^{(1)}(t))
\end{align*}
for $m=1,2,\hdots,k$. Based on the proof of Lemma~\ref{lem:results_deg_2} we state the following 

\begin{hyp}\label{hyp}
Within the setting of Subsection~\ref{ssec:sr}, consider a sub-Riemannian abnormal extremal corresponding to a control $u\in L^2([0,T],\R^l)$. If the assumptions of \cite[Thm 1.1]{Boarotto_Monti_Socionovo_2022} are satisfied then system \eqref{eqn:aff_ctr_sys_polynomial} is not-controllable at zero. 
\end{hyp}

If the above hyphothesis would be true, we could get higher-order Goh conditions from Sussmann's-Jurdvevic's non-controllability conditions in essentially the same way as was done in the proof of Lemma~\ref{lem:results_deg_2}. (Perhaps some other conditions could also be computable following the lines of \cite{Agrachev_2023}.) To back this up, after quite heavy calculations, it is possible to check that  for all possible indices $i,j,s\in\{1,2,\hdots,l\}$ we have
\begin{align*} 
    \Z{1}_i=&(0,0,\lin{1}(t)[X_i]),\\
    [\Z{2}_i,\Z{2}_j]=&(0,0,\lin{1}(t)[X_i,X_j])\quad\text{and}\\
    [\Z{3}_i,[\Z{3}_j,\Z{3}_s]]=&(0,0,\lin{1}(t)\left[[X_i,[X_j,X_s]]\right])+\\
    & (0,0,\lin{1}(t)\left[A(t)\lin{2}(t)[X_i,X_j-X_s\right])+\\
    &(0,0,\lin{1}(t)\left[\D X_s[A(t)\lin{2}(t)[X_i,X_j]]-\D X_j[A(t)\lin{2}(t)[X_i,X_s]]\right]),
\end{align*}
where $A(t)$ denotes the inverse of the map $\lin{1}(t)$. Hence the existence of the covector $(0,0,\psi_0)\in (\R^n\times\R^n\times\R^n)^\ast$ annihilating all these vector fields will imply the existence of a Pontryagin covector (cf. Lemmta~\ref{lem:results_deg_1} and \ref{lem:results_deg_2}) $\phi(t):=\lin{1}(t)^\ast\psi_0$ such that $\<\phi(t),X_i>$ and $\<\phi(t),[X_i,X_j]>$ vanish. To get conditions $\<\phi(t),[X_i,[X_j,X_k]]>=0$, we should additionally assume that $\psi_0$ annihilates the remaining two terms in the triple Lie bracket. We believe that this condition will follow from the assumptions of \cite[Thm 1.1]{Boarotto_Monti_Socionovo_2022}, i.e. the extremal being of corank 1 and vanishing of the second (intrinsic) derivative of the end-point map at $u$. We shall study this topic in more detail in a future publication.






\bibliographystyle{amsalpha}
\bibliography{bibl}

\end{document}